\documentclass{amsart}

\usepackage{amsmath, amsthm, amscd, amsfonts, amssymb, graphicx, color}
\usepackage{datetime}

\usepackage{microtype}

\newcommand{\supast}{{${^\ast}$}}
\newcommand{\ulb}{{\textup{(}}}
\newcommand{\urb}{{\textup{)}}}

\newcommand{\T}{{\mathbb{T}}}
\newcommand{\C}{{\mathbb{C}}}
\newcommand{\N}{{\mathbb{N}}}
\newcommand{\Z}{{\mathbb{Z}}}

\newcommand{\rest}[1]{{\restriction_{#1}}}
\newcommand{\Ker}{{\textup{Ker\,}}}
\newcommand{\supp}{{\textup{supp\,}}}

\newcommand{\Cstar}{{$C^\ast$}}

\newcommand{\Zconv}[1]{{#1}\in\Z}

\newcommand{\sumvar}[1]{\sum_{\Zconv{#1}}}
\newcommand{\sumk}{\sumvar{k}}
\newcommand{\suml}{\sumvar{l}}

\newcommand{\sumn}{\sumvar{n}}

\newcommand{\bigcupvar}[1]{\bigcup_{\Zconv{#1}}}

\newcommand{\bigcupn}{\bigcupvar{n}}

\newcommand{\rep}{{\pi}}

\newcommand{\ai}{algebraically irreducible }
\newcommand{\ains}{algebraically irreducible}
\newcommand{\Ai}{Algebraically irreducible }

\newcommand{\alg}{A}
\newcommand{\na}{A}
\newcommand{\ba}{A}
\newcommand{\vs}{E}
\newcommand{\ns}{E}
\newcommand{\bs}{E}
\newcommand{\hs}{H}

\newcommand{\linear}{\mathcal L}
\newcommand{\bounded}{\mathcal B}

\DeclareFontFamily{OT1}{pzc}{}
\DeclareFontShape{OT1}{pzc}{m}{it}{<-> s * [1.15] pzcmi7t}{}
\DeclareMathAlphabet{\mathscr}{OT1}{pzc}{m}{it}

\newcommand{\hull}{{\mathscr h}}
\newcommand{\kernel}{{\mathscr k}}
\newcommand{\Kernel}{{\mathscr K\,}}

\newcommand{\topspace}{{X}}
\newcommand{\pt}{x}
\newcommand{\homeo}{{\sigma}}
\newcommand{\dynsys}{(\topspace, \homeo)}
\newcommand{\dynsysshort}{{\Sigma}}
\newcommand{\per}{{\textup{Per}}}
\newcommand{\aper}{{\textup{Aper}}}
\newcommand{\aperpoints}{{\aper(\homeo)}}
\newcommand{\perpoints}{{\per(\homeo)}}
\newcommand{\perpointsvariable}[1]{{{\per}_{#1}(\homeo)}}

\newcommand{\pperpoints}{\perpointsvariable{p}}
\newcommand{\orb}{\mathfrak{o}}
\newcommand{\corb}{\overline{\orb}}
\newcommand{\orbset}{\mathfrak{O}}
\newcommand{\corbset}{\overline{\orbset}}
\newcommand{\corbsetperpoints}{\corbset_{\perpoints}}
\newcommand{\corbsetaperpoints}{\corbset_{\aperpoints}}
\newcommand{\orbspace}[1]{#1/\Z}

\newcommand{\prim}{P}
\newcommand{\sspace}{{\mathit \Pi}_\lone}
\newcommand{\sspaceset}[1]{\sspace\left({#1}\right)}
\newcommand{\sspaceinfdim}{\sspaceset{\infty}}

\newcommand{\cspace}[1]{C({#1})}
\newcommand{\coeffalg}{\cspace{\topspace}}

\newcommand{\inducedset}{\topspace_\rep}
\newcommand{\inducedcoeffalg}{\cspace{\inducedset}}
\newcommand{\inducedhomeo}{\homeo_{\inducedset}}
\newcommand{\inducedsystem}{(\inducedset,\inducedhomeo)}
\newcommand{\inducedsystemshort}{\Sigma_\rep}
\newcommand{\induceddelta}{\delta_{\inducedset}}
\newcommand{\inducedlone}{\ell^1(\inducedsystemshort)}
\newcommand{\inducedrest}{\mathcal R_{\inducedset}}
\newcommand{\inducedquotientmap}{\mathcal Q_{\inducedset}}

\newcommand{\lone}{{\ell^1(\dynsysshort)}}
\newcommand{\loneZ}{{\ell^1(\Z)}}
\newcommand{\loneS}{{\ell^1(\dynsysshort_S)}}
\newcommand{\cstar}{{C^\ast(\dynsysshort)}}
\newcommand{\loneelement}[1]{{\sumvar{#1} f_{#1}\delta^{#1}}}

\theoremstyle{plain}

\newtheorem{theorem}{Theorem}[section]
\newtheorem{proposition}[theorem]{Proposition}
\newtheorem{lemma}[theorem]{Lemma}
\newtheorem{corollary}[theorem]{Corollary}

\theoremstyle{definition}

\newtheorem{example}[theorem]{Example}
\newtheorem{remark}[theorem]{Remark}

\numberwithin{equation}{section}

\begin{document}


\title [Algebraically irreducible representations and structure space]{Algebraically irreducible representations and structure space of the Banach algebra associated with a topological dynamical system}

\author{Marcel de Jeu}
\address{Marcel de Jeu, Mathematical Institute, Leiden University, P.O.\ Box 9512, 2300 RA Leiden, the Netherlands}
\email{mdejeu@math.leidenuniv.nl}

\author{Jun Tomiyama}
\address{Jun Tomiyama, Department of Mathematics, Tokyo Metropolitan University, Minami-Osawa, Hachioji City, Japan}
\email{juntomi@med.email.ne.jp}




\begin{abstract}
If $\topspace$ is a compact Hausdorff space and $\homeo$ is a homeomorphism of $\topspace$, then a Banach algebra $\lone$ of crossed product type is naturally associated with this topological dynamical system $\dynsysshort=(\topspace,\homeo)$. If $\topspace$ consists of one point, then $\lone$ is the group algebra of the integers.

We study the \ai representations of $\lone$ on complex vector spaces, its primitive ideals, and its structure space. The finite dimensional \ai representations are determined up to algebraic equivalence, and a sufficiently rich family of infinite dimensional \ai representations is constructed to be able to conclude that $\lone$ is semisimple. All primitive ideals of $\lone$ are selfadjoint, and $\lone$ is Hermitian if there are only periodic points in $\topspace$. If $\topspace$ is metrizable or all points are periodic, then all primitive ideals arise as in our construction. A part of the structure space of $\lone$ is conditionally shown to be homeomorphic to the product of a space of finite orbits and $\T$.  If $\topspace$ is a finite set, then the structure space is the topological disjoint union of a number of tori, one for each orbit in $\topspace$. If all points of $\topspace$ have the same finite period, then it is the product of the orbit space $\orbspace{\topspace}$ and $\T$. For rational rotations of $\T$, this implies that the structure space is homeomorphic to $\T^2$.
\end{abstract}

\maketitle

\section {Introduction and overview}\label{sec:intro}

If $\topspace$ is a compact Hausdorff space and $\sigma$ is a homeomorphism of $\topspace$, then there is a Banach algebra $\lone$ of crossed product type associated with the dynamical system $\dynsysshort=(\topspace,\homeo)$. It is an involutive algebra, and there is a significant amount of literature on the relation between the properties of the enveloping \Cstar-algebra $\cstar$ of $\lone$ and those of the dynamical system. The algebra $\lone$ itself, however, is far less well studied, even though it is arguably more naturally associated with $\dynsysshort$ than $\cstar$, the construction of which takes one extra step. The investigation of $\lone$, which is an algebra with a more complicated structure than $\cstar$, has been taken up in \cite{DST} and has been continued in \cite{JTStudia} and \cite{JTBJMA}.

The present paper is a further step in the study of $\lone$. It fits into what seems to be an emerging  line of research where Banach algebras of crossed product (or related) type are considered that are associated with (abstract) dynamical systems, but that are not \Cstar-algebras or closed subalgebras of \Cstar-algebras. We refer to \cite{DK2008, DK2009, DK2011, DdJW, dJMW, dJM, Phillips2012, Phillips2013a, Phillips2013b} as examples of this development. These new algebras present an extra challenge compared to \Cstar-algebras, because the latter with their rigidity properties are still reasonably manageable, and have a relatively uncomplicated\textemdash though still far from trivial\textemdash structure. As an example, the algebra $\ell^1(\Z)$\textemdash our algebra $\lone$ reduces to this algebra when $\topspace$ consists of one point\textemdash has closed non-selfadjoint ideals, whereas this is of course no longer true for its enveloping \Cstar-algebra $C(\T)$.

In this paper, we concentrate on the \ai representations of $\lone$ on complex vector spaces, the primitive ideals, and the structure space of $\lone$. Hence there is no topology on the representation space involved, although the fact that a topology can always be brought into play (see Theorem~\ref{thm:material_on_algebraically_irreducible_representations}) will play an important part in the proofs. This is contrary to what sometimes seems to have become the main objective (in particular for involutive Banach algebras), namely to study topologically irreducible (\supast-)representations. In many basic papers, including \cite{PT}, this is done without further comment, and also the present authors have used this definition of an irreducible representation in \cite{JTBJMA}. The fact that for $\cstar$-algebras there is no difference (see Theorem~\ref{thm:cstar_reps}) will have encouraged this tendency. In the present paper, however, we return to the purely algebraic viewpoint. This gives the correct definition of a primitive ideal that enables one to introduce the hull-kernel topology on the set of primitive ideals.

The reader who is familiar with the formulas used to define representations of $\lone$ in \cite{DST, JTStudia, JTBJMA} will notice a clear similarity with the formulas in the present paper. Although these formulas have certainly been an inspiration, the similarity does not go much further than that, because in a purely algebraic context we need other techniques than in the previous papers. For example, it is not so difficult to determine the finite dimensional \ai \supast-representations of $\lone$, since there is a theory of states available in this Hilbert space context, but to show that these actually exhaust the finite dimensional \ai representations on complex vector spaces up to equivalence is another matter. Likewise, with every aperiodic point we shall associate a representation of $\lone$ on $\ell^p(\Z)$  for every $p\in[1,\infty]$. It is easy to show that this representation is topologically irreducible if $p\in(0,\infty)$ and not topologically irreducible if $p=\infty$ (see Proposition~\ref{prop:topological_irreducibility}), but proving that it is algebraically irreducible if $p=1$ is more demanding (see Proposition~\ref{prop:one_algebraically_irreducible}).

This paper is organised as follows.

In Section~\ref{sec:preliminaries}, we establish notation, collect non-trivial key material on \ai representations of Banach algebras, introduce our Banach algebra $\lone$, and establish basic results on dynamical systems that are induced by non-zero homomorphisms from $\lone$ into a normed algebra.

Section~\ref{sec:representations} contains the description of all finite dimensional \ai representations of $\lone$ up to algebraic equivalence (see Theorem~\ref{thm:description_finite_dimensional_representations}). For $p\in[1,\infty]$, the representations of $\lone$ on $\ell^p(\Z)$ associated with aperiodic points are introduced and investigated from the viewpoint of equivalence (see Propositions~\ref{prop:equivalences} and~\ref{prop:top_equivalences}) and algebraic and topological irreducibility (see Theorem~\ref{thm:irreducibility_infinite_dimensional}). With the primary algebraic goal of this paper in mind, we could have restricted ourselves to the algebraically irreducible representations on $\ell^1(\Z)$ and their algebraic equivalences, but it seemed less than satisfactory not to present the complete picture.

Section~\ref{sec:primitive_ideals} combines the \ai representations from Section~\ref{sec:representations} with the technique of induced dynamical systems from Section~\ref{sec:preliminaries}. It is shown that, even though we do not generally know them all, every primitive ideal of $\lone$ is selfadjoint (see Theorem~\ref{thm:primitive_ideals_are_selfadjoint}). Furthermore, $\lone$ is a Hermitian Banach algebra if all points of $\topspace$ are periodic (see Theorem~\ref{thm:Hermitian}). If $\topspace$ is metrizable, then, even though we do not generally know all infinite dimensional \ai representations, we can still show that all primitive ideals can be obtained from Section~\ref{sec:preliminaries} (see Theorem~\ref{thm:exhaustive}). This section also contains\textemdash partly as a prelude to Section~\ref{sec:structure_space}\textemdash a more detailed investigation of the primitive ideals originating from Section~\ref{sec:representations}. There are already enough of these to conclude that $\lone$ is semisimple (see Theorem~\ref{thm:semisimple}).

The final Section~\ref{sec:structure_space} concentrates on the structure space of $\lone$, i.e.\  on the set of primitive ideals of $\lone$ in the hull-kernel topology. Several results of Section~\ref{sec:primitive_ideals} can be interpreted in this context (see Theorem~\ref{thm:basic_topology}). The main goal of this section is the description\textemdash under conditions\textemdash of parts of this structure space as topological products of spaces of finite orbits and $\T$ (see Theorem~\ref{thm:main_homeomorphism}). If $\topspace$ is a finite set, then the structure space is the topological disjoint union of a number of tori, one for each orbit in $\topspace$. If all points of $\topspace$ are periodic with the same period, then it is homeomorphic to the product of the orbit space $\orbspace{\topspace}$ and $\T$. For rational rotations of $\T$, this implies that the structure space is homeomorphic to $\T^2$. The methods in Section~\ref{sec:structure_space} could conceivably be adapted to yield similar results for the structure space of $\cstar$, but further research is needed to explore this perspective.

\section{Preliminaries}\label{sec:preliminaries}

This section contains the necessary preliminary definitions and results.

We start by introducing some conventions and terminology for representations. The latter, which we give in detail since there are various different terminologies in use, is consistent with that in \cite{Pbook}, with the exception that\textemdash to prevent any misunderstanding\textemdash we write `\ains' where \cite{Pbook} uses `irreducible'.

All vector spaces in this paper are complex. Algebras are not necessarily unital. Ideals of an algebra are two-sided; ideals of normed algebras are not necessarily closed. A \emph{representation} of an algebra $\alg$ on a vector space $\vs$ is a not necessarily unital homomorphism $\rep:A\to\linear(\vs)$ into the linear operators $\linear(\vs)$ on $E$. It is \emph{\ai}if $\rep(\ba)(\vs)\neq\{0\}$ and the only invariant subspaces of $\vs$ are $\{0\}$ and $\vs$. An \ai representation of a unital algebra is necessarily unital. If $\ns$ is a normed space, then a \emph{normed representation} of an algebra $\alg$ on $\ns$ is a representation $\rep$ of $\na$ on $\ns$ such that $\pi(\na)\subset\bounded(\ns)$, where $\bounded(\ns)$ denotes the bounded operators on $\ns$. A normed representation is \emph{topologically irreducible} if $\rep(\na)(\ns)\neq\{0\}$ and the only closed invariant subspaces of $\vs$ are $\{0\}$ and $\vs$.  A topologically irreducible normed representation of a unital algebra is necessarily unital. A normed representation of a normed algebra $\na$ on a normed space $\ns$ is \emph{continuous} if $\rep:\na\to\bounded(\ns)$ is continuous, and it is \emph{contractive} if $\rep:\na\to\bounded(\ns)$ is contractive. The notions \emph{\supast-representation} and \emph{algebraic equivalence}, \emph{topological equivalence}, \emph{isometric equivalence}, and \emph{unitary equivalence} of representations are self-explanatory. A \supast-representation of a Banach algebra with isometric involution on a Hilbert space is automatically contractive.

Next, we collect some non-trivial facts on \ai representations of Banach algebras.

\begin{theorem}\label{thm:material_on_algebraically_irreducible_representations} Let $\ba$ be a Banach algebra and let $\vs$ be a vector space. Suppose that $\rep:\ba\to\linear(\vs)$ is an \ai representation. Then:
\begin{enumerate}
\item The algebra of intertwining operators on $\vs$ consists of the complex multiples of the identity.
\item $\vs$ has a unique Banach space topology relative to which $\rep$ is normed. There exists a norm inducing this topology such that $\rep$ is a contractive representation.
\item An algebraic equivalence between two \ai normed representations of $\ba$ on Banach spaces is a topological equivalence.
\end{enumerate}
\end{theorem}

Part (1) follows from \cite[Corollary~25.3.(i) and Theorem~14.2]{BD}, where we use the fact that we are working over $\C$.
The first part of (2) is \cite[Corollary~4.2.16.(a)]{Pbook}, and the second part follows from an inspection of the proof of \cite[Lemma~25.2]{BD}, or as a special case of \cite[Theorem~4.2.7]{Pbook}.
Part (3) is \cite[Corollary~4.2.16.(b)]{Pbook}.

 The fact that every \ai representation of a Banach algebra can be viewed as a continuous (even contractive) representation, as asserted in part (2), will be used repeatedly. This possibility is a consequence of the fact that maximal modular left ideals of a Banach algebra are closed.

The next result may well underlie the fact that in parts of the literature the `irreducibility' of a normed representation stands for what is `topological irreducibility' in our terminology. The `if' part is Kadison's result (see \cite{K} or \cite[Corollary~2.8.4]{D}); the `only if' part follows from \cite[Corollary~2.9.6.(i)]{D}.

\begin{theorem}\label{thm:cstar_reps}
 Let $\ba$ be a \Cstar-algebra and let $\rep$ be a representation of $\ba$ on a vector space $E$. Then $\pi$ is \ai if and only if $\vs$\ can be supplied with the structure of a Hilbert space such that $\rep$ is a topologically irreducible \supast-representation of $\ba$ on $\vs$.
\end{theorem}

We now turn to the dynamical system and its associated Banach algebra.

Throughout this paper, $\topspace$ is a non-empty compact Hausdorff space and $\homeo:\topspace\to\topspace$ is a homeomorphism. Hence $\Z$ acts on $\topspace$, and we write $\dynsysshort=\dynsys$ for this topological dynamical system. We let $\aperpoints$ and $\perpoints$ denote the aperiodic and the periodic points of $\homeo$, respectively. We say that $\dynsys$ is \emph{topologically free} if $\aperpoints$ is dense in $\topspace$, and that it is \emph{free} if the $\Z$-action is free, i.e.\  if $\perpoints=\emptyset$. It is \emph{topologically transitive} if $\bigcupn \homeo^n(V)$ is dense in $\topspace$ for every non-empty open subset $V$ of $\topspace$. For every integer $p\geq 1$, let $\pperpoints$ be the set of points with an orbit of $p$ elements. A subset $S$ of $\topspace$ is
\emph{invariant} if it is invariant under the $\Z$-action, i.e.\  if $\homeo(S)=S$. If $S$ is invariant, then so are its closure and interior. The sets $\aperpoints$, $\perpoints$ are invariant, as are the sets $\pperpoints$ for every integer $p\geq 1$. We shall write $\orb$ for a general finite or infinite orbit, and $\corb$ for the closure of an orbit.

The involutive algebra of continuous (complex-valued) functions on $\topspace$ is denoted by $\coeffalg$, and we write $\alpha$ for the involutive automorphism of $\coeffalg$ induced by $\homeo$, defined by $\alpha(f) = f \circ \homeo^{-1}$ for $f \in\coeffalg$. Via $n \mapsto \alpha^n$, $\Z$ acts on $\coeffalg$.

With $\Vert\cdot\Vert$ denoting the supremum norm on $\coeffalg$, we let
$$
\lone = \ell^1(\Z,\coeffalg)=\left\{a: \mathbb{Z} \to\coeffalg : \Vert a \Vert := \sumn \Vert a(n)\Vert < \infty\right\}.
$$
We supply $\lone$ with the usual twisted convolution as multiplication, defined by
$$
(aa^\prime) (n) = \sumk a(k) \cdot \alpha^k (a^\prime(n-k))
$$
for $n\in\Z$ and $a, a^\prime \in \lone$, and define an involution on $\lone$ by
$$
a^* (n) = \overline{\alpha^n (a(-n))}
$$
for $n\in\Z$ and $a\in\lone$. Thus $\lone$ becomes a unital Banach $\sp\ast$-algebra with isometric involution, and we call $\lone$ \emph{the Banach algebra associated with $\dynsysshort$}. If $\topspace$ consists of one point, then $\lone$ is the group algebra $\loneZ$ of the integers.

A convenient way to work with $\lone$ is provided by the following. For $n,m \in \mathbb{Z}$, let
\begin{equation*}
 \chi_{\{n\}} (m) =
 \begin{cases}
   1 &\text{if }m =n;\\
   0 &\text{if }m \neq n,
 \end{cases}
\end{equation*}
where the constants denote the corresponding constant functions in $\coeffalg$. Then $\chi_{\{0\}}$ is the identity element of $\lone$. Let $\delta = \chi_{\{1\}}$; then $\chi_{\{-1\}}=\delta^{-1}=\delta^*$. If we put $\delta^0=\chi_{\{0\}}$, then $\delta^n = \chi_{\{n\}}$ for all $n \in \mathbb{Z}$. We may view $\coeffalg$ as a closed abelian \supast-subalgebra of $\lone$, namely as $\{a_0 \delta^0 \, : \, a_0 \in \coeffalg\}$. If $a \in \lone$, and if we write $f_n$ for  $a(n)$ as a more intuitive notation, then $a= \loneelement{n}$ and $\Vert a \Vert=\sumn \Vert f_n\Vert<\infty$.
In the rest of this paper, we shall constantly use this series representation $a= \loneelement{n}$ of an arbitrary element $a\in\lone$, with uniquely determined $f_n\in\coeffalg$ for $n\in\Z$. Thus $\lone$ is generated, as a unital Banach algebra, by an isometrically isomorphic copy of $\coeffalg$ and the elements $\delta$ and $\delta^{-1}$, subject to the relation $\delta f \delta^{-1}=\alpha (f)= f\circ\homeo^{-1}$ for $f \in\coeffalg$. The isometric involution is determined by $f^*=\overline f$ for $f\in\coeffalg$ and by $\delta^*=\delta^{-1}$.

\medskip

We continue our preparations by including some material on Banach algebras associated with non-empty invariant closed subsets and\textemdash this is the actual purpose\textemdash with homomorphisms.

For a non-empty closed invariant subset $S$ of $\topspace$, let
\[
\Kernel(S) = \left\{\sumn f_n \delta^n \in\lone : f_n\rest{S}=0\textup{ for all }n\in\Z\right\}
\]
be the closed ideal of $\ell^1(\dynsysshort)$ that is generated by $\{f\in\coeffalg : f\rest{S}=0\}$. It is proper and selfadjoint.
Since $S$ is invariant, $\dynsysshort_S:=(S,\homeo\rest{S})$ is a topological dynamical system in its own right; hence there is an associated Banach algebra $\loneS$. Elements of this algebra can be written as absolutely convergent series $\sumn g_n\delta_S^n$, where $g_n\in C(S)$  for $n\in\Z$.

\begin{lemma}\label{lem:restricted_system}
Let $S\subset \topspace$ be a non-empty invariant closed subset, and define $\mathcal R:\lone\to\loneS$ by
\[
\mathcal{R_S}\left(\sumn f_n\delta^n\right)=\sumn f_n\rest{S}\,\delta_S^n
\]
for $\sumn f_n\delta^n\in\lone$.
Then:
\begin{enumerate}
\item $\mathcal R_S$ is a surjective unital contractive \supast-homomorphism;
\item $\Ker(\mathcal R_S)=\Kernel(S)$ is a proper closed selfadjoint ideal of $\ell^1(\dynsysshort)$;
\item If $\mathcal Q_S:\lone\to\lone/\Kernel(S)$ denotes the unital quotient \supast-homomorphism, then the induced map $\mathcal R_S^\prime: \lone/\Kernel(S)\to\loneS$ such that $\mathcal R_S^\prime\circ\mathcal Q_S=\mathcal R_S$ is an isometric \supast-isomorphism.
\end{enumerate}
\end{lemma}

\begin{proof}
As a consequence of Tietze's extension theorem, the canonical map from $\coeffalg/\{f\in\coeffalg : f\rest{S}=0\}$ to $C(S)$ is a \supast-isomorphism of \Cstar-algebras. Hence it is isometric. Therefore, if $g\in C(S)$ and $\varepsilon>0$ are given, there exists $f\in\coeffalg$ such that $f\rest{S}=g$ and $\Vert f\Vert < \Vert g\Vert + \varepsilon$. This implies that $\mathcal R_S$ is surjective. The rest in (1) is clear, as is (2), and the isometric nature of the induced map $\mathcal R_S^\prime$ in (3) follows again from the above extension property of elements of $C(S)$.
\end{proof}

The above construction can be put to good use when studying homomorphisms, as follows.
Let $\rep:\lone\to A$ be a non-zero not necessarily unital continuous homomorphism from $\lone$ into a not necessarily unital normed algebra $A$. If $\rep(\coeffalg)=\{0\}$, then $\rep(1)=0$, implying that $\rep=0$. Since this is not the case, $\{f\in\coeffalg : \rep(f)=0\}$ is a proper $\alpha$-invariant closed ideal of $\coeffalg$. There exists a unique closed subset $\inducedset$ of $\topspace$ such that, for $f\in\coeffalg$, $\rep(f)=0$ if and only if $f\rest{\inducedset}=0$; we see that $\inducedset$ is non-empty and $\homeo$-invariant.  The dynamical system $\inducedsystemshort:=\inducedsystem$ is called the \emph{dynamical system induced by} $\rep$. The continuity of $\rep$ implies that $\Kernel(\inducedset)\subset\Ker (\rep)$; hence we have a continuous homomorphism $\tilde\rep: \lone/\Kernel(\inducedset)\to A$ such  that $\rep=\inducedquotientmap\circ\tilde \rep$, where $\inducedquotientmap:\lone\to\lone/\Kernel(\inducedset)$ is the quotient map. Then $\Vert \tilde\rep\Vert=\Vert\rep\Vert$. On the other hand, Lemma~\ref{lem:restricted_system} shows that the map $\inducedrest^\prime:\lone/\Kernel(\inducedset)\to\ell^1(\inducedsystemshort)$ such that $\inducedrest^\prime\circ\inducedquotientmap=\inducedrest$ is an isometric \supast-isomorphism. If we let $\rep^\prime=\tilde\rep\circ\left(\inducedrest^\prime\right)^{-1}$, then $\rep^\prime:\ell^1(\inducedsystemshort)\to A$ is continuous, $\Vert\rep^\prime\Vert=\Vert\tilde\rep\Vert=\Vert\rep\Vert$, and $\rep=\rep^\prime\circ\mathcal R$. The following is now clear.

\begin{proposition}\label{prop:induced_system}
Let $\rep:\lone\to A$ be a non-zero not necessarily unital continuous homomorphism from $\lone$ into a not necessarily unital normed algebra $A$. Then $\{f\in\coeffalg : \rep(f)=0\}=\{f\in\coeffalg : f\rest{\inducedset}=0\}$ for a unique closed subset $\inducedset$ of $\topspace$. This $\inducedset$ is non-empty and invariant; hence it yields a dynamical system $\inducedsystemshort=(\inducedset,\homeo\rest{\inducedset})$. Define $\inducedrest:\lone\to\ell^1(\inducedsystemshort)$ by
\[
\inducedrest\left(\sumn f_n\delta^n\right)=\sumn f_n\rest{\inducedset}\,\delta_{\inducedset}^n
\]
for $\sumn f_n\delta^n\in\lone$. Then:
\begin{enumerate}
\item $\inducedrest$ is a surjective unital contractive \supast-homomorphism;
\item There exists a unique map $\rep^\prime:\ell^1(\inducedsystemshort)\to A$ such that $\rep=\rep^\prime\circ\inducedrest$. This $\rep^\prime$ is a continuous homomorphism and $\Vert\rep^\prime\Vert=\Vert\rep\Vert$. If $A$ is unital, then $\rep$ is unital precisely when $\rep^\prime$ is unital. If $A$ is involutive, then $\rep$ is involutive precisely when $\rep^\prime$ is involutive;
\item $\rep^\prime$ is injective on $\inducedcoeffalg$;
\item If $A=\bounded(\ns)$\ for a normed space $\ns$, then $\rep$ is an algebraically \ulb resp.\ topologically\urb\ irreducible representation of $\lone$ on $\ns$ if and only if $\rep^\prime$ is an algebraically \ulb resp.\ topologically\urb\ irreducible representation of $\ell^1(\inducedsystemshort)$ on $\ns$.
\end{enumerate}
\end{proposition}

\begin{remark}\label{rem:can_be_assigned}
If $\rep:\lone\to\linear(\vs)$ is an \ai representation of $\lone$ on a vector space $\vs$, then, as noted in Theorem~\ref{thm:material_on_algebraically_irreducible_representations}, we may assume that $\vs$ is a normed space (even that it is a Banach space) and that $\pi:\ba\to\bounded(\vs)$ is continuous (even that it is contractive). Hence we can assign an induced dynamical system $\inducedsystemshort$ and corresponding Banach algebra $\ell^1(\inducedsystemshort)$ to every \ai representation of $\lone$, and Proposition~\ref{prop:induced_system} applies. We shall exploit this several times in Section~\ref{sec:primitive_ideals}.
\end{remark}

The following result is the cornerstone when showing that the primitive ideals corresponding to the infinite dimensional \ai representations of $\lone$ are selfadjoint (see the proof of Theorem~\ref{thm:primitive_ideals_are_selfadjoint}). It relies on one of the main results of \cite{DST}: the commutant $\coeffalg^\prime$ of $\coeffalg$ in $\lone$ has non-zero intersection with every non-zero closed ideal of $\lone$ (see \cite[Theorem~3.7]{DST}).

\begin{proposition}\label{prop:self_adjoint}
Let $\rep$ be a non-zero not necessarily unital continuous homomorphism from $\lone$ into a not necessarily unital normed algebra $\ba$ such that its induced dynamical system $\inducedsystemshort$ is topologically free. As in Proposition~\ref{prop:induced_system}, let $\rep^\prime:\inducedlone\to\ba$ be the continuous homomorphism such that $\pi=\rep^\prime\circ\inducedrest$. Then $\rep^\prime$ is injective on $\ell^1(\inducedsystemshort)$. Consequently, $\Ker(\rep)=\Ker(\inducedrest)=\Kernel(\inducedset)$ is the closed ideal of $\lone$ that is generated by $\{f\in\coeffalg : f\rest{\inducedset}=0\}$. In particular, $\Ker(\rep)$ is selfadjoint.
\end{proposition}

\begin{proof}
To see that $\rep^\prime$ is injective, assume that $\Ker(\rep^\prime)\neq\{0\}$. Then the result mentioned preceding the theorem, applied to $\ell^1(\inducedsystemshort)$, implies that $\Ker(\rep^\prime)\cap C(\inducedset)^\prime\neq\{0\}$, where $C(\inducedset)^\prime$ is the commutant of $\inducedcoeffalg$ in $\ell^1(\inducedsystemshort)$. However, since $\Sigma$ is topologically free, \cite[Proposition~3.1]{DST} implies that $C(\inducedset)^\prime=C(\inducedset)$. Hence $\Ker(\rep^\prime)\cap C(\inducedset)\neq\{0\}$. But this contradicts the injectivity of $\rep^\prime$ on $C(\inducedset)$ in Proposition~\ref{prop:induced_system}. Therefore we must have $\Ker(\rep^\prime)=\{0\}$, and hence $\Ker(\rep)=\Ker(\inducedrest)$. The rest is clear.
\end{proof}

We conclude our preparations with a few elementary topological results for which we are not aware of a reference. They are needed in the proof of Theorems~\ref{thm:technical_main_theorem} and~\ref{thm:main_homeomorphism}.

\begin{lemma}\label{lem:topology} Let $S$ be a topological space.
 \begin{enumerate}
  \item Suppose that $S=\bigcup_{i=1}^n S_i$ is the finite disjoint union of subsets $S_i$, where each $S_i$ is a closed subset of $S$ that is a compact Hausdorff space in the induced topology. Then the topological space $S$ is a compact Hausdorff space, and it is the disjoint union $\bigsqcup_{i=1}^n S_i$ of the topological spaces $S_i$ carrying their induced topologies.
  \item Suppose that $S=\bigsqcup_{i\in I} S_i$ is the arbitrary disjoint union of topological spaces $S_i$, and that $T$ is a topological space. Then $S\times T=\bigsqcup_{i\in I}\left(S_i \times T\right)$ as a disjoint union of topological spaces.
  \item Suppose that $S=\bigsqcup_{i\in I} S_i$ is the arbitrary disjoint union of topological spaces $S_i$. Let $\sim$ be an equivalence relation on $S$ such that each equivalence class in $S$ is entirely included in one \ulb unique\urb\ $S_i$. Then $S/\sim=\bigsqcup_{i\in I} \left(S_i/\sim\right)$ as a disjoint union of topological spaces.
  \end{enumerate}
\end{lemma}

\begin{proof}
For part (1), we note that each $S_i$ is open. Combining this with the fact that the $S_i$ are Hausdorff, one sees that $S$ is Hausdorff. It is now already clear that $S$ is a compact Hausdorff space. Since the canonical map from the disjoint union of topological spaces $\bigsqcup_{i=1}^n S_i$ to $S$ is a continuous bijection between a compact space and a Hausdorff space, it is a homeomorphism. This proves (1). Parts (2) and (3) are completely elementary.
\end{proof}

\section{Representations associated with points}\label{sec:representations}

In this section, we study representations of $\lone$ that are naturally associated with the points of $\topspace$,  and their algebraic and (when applicable) topological equivalence. To each periodic point corresponds a family (parameterized by $\T$) of finite dimensional \ai representations. Each of the pertinent representation spaces can be supplied with a Hilbert space structure such that the representation is then a \supast-representations, and, taken together, these representations exhaust the \ai representations up to algebraic equivalence (see Theorem~\ref{thm:description_finite_dimensional_representations}). The representations associated with infinite orbits have $\ell^p(\Z)$  for $p\in[1,\infty]$ as representation spaces, and also for these representations we can resolve the algebraic equivalence and the algebraic or topological irreducibility questions (see Proposition~\ref{prop:equivalences} and Theorem~\ref{thm:irreducibility_infinite_dimensional}).

\subsection{Preparations}\label{subsec:preliminaries_representations}

In order not to interrupt the main exposition, we formulate two preparatory results on representations of $\lone$ in this subsection that will be used a number of times.

Part (1) of the following result follows from more general results for covariant representations of Banach algebra dynamical systems (see \cite[Theorem~5.20]{dJMW}), applied to the case $(\coeffalg,\Z,\alpha)$ at hand. However, once one notes that $||\delta^n||=1$ for all $n\in\Z$, both part (1) and (2) can also be derived by elementary arguments for the discrete group $\Z$; the proof is left to the reader.

\begin{lemma}\label{lem:how_to_build_representations}\quad
\begin{enumerate}
\item  Let $\rep:\lone\to\bounded(E)$ be a unital continuous representation of $\lone$ on a normed space $\bs$. Put $T=\rep(\delta)$. Then the restriction $\rho:=\rep\rest{\coeffalg}:\coeffalg\to\bounded(E)$ of $\rep$ to $\coeffalg$ is a unital continuous representation of $\coeffalg$ on $\bs$, $T$ is invertible in $\bounded(\bs)$, $\rho(\alpha(f))=T\rho(f)T^{-1}$ for $f\in\coeffalg$,
and there exists $M\geq 0$ such that $||T^n||\leq M$ for all $n\in\Z$. If $\vs$ is a Banach space, then, conversely,
if $\rho$ and $T$ are given satisfying these four properties, then there
is a unique unital continuous representation $\rep$ of $\lone$ on $\bs$ such that its restriction to $\coeffalg$ is $\rho$ and $T=\rep(\delta)$.
\item Let $\rep:\lone\to\bounded(\hs)$ be a unital \supast-representation of $\lone$ on a Hilbert space $\hs$. Put $T=\rep(\delta)$. Then the restriction $\rho:=\rep\rest{\coeffalg}:\coeffalg\to\bounded(\hs)$ of $\rep$ to $\coeffalg$ is a unital \supast-representation of $\coeffalg$ on $\hs$, $T$ is unitary, and $\rho(\alpha(f))=T\rho(f)T^{-1}$ for $f\in\coeffalg$.
Conversely, if $\rho$ and $T$ are given satisfying these three properties,
then there is a unique unital \supast-representation $\rep$ of $\lone$ on $\hs$ such that its restriction to $\coeffalg$ is $\rho$ and $T=\rep(\delta)$.
\end{enumerate}
\end{lemma}

We shall need the following result when studying the common eigenspaces for $\coeffalg$ that are associated with representations of $\lone$.

\begin{lemma}\label{lem:orbits_eigenvalues}
Let $\rep:\lone\to\linear(\vs)$ be a unital representation of $\lone$ on a
vector space $\vs$. For $\pt\in \topspace$, let
\begin{equation}\label{eq:common_eigenspace}
\vs_\pt=\{e\in \vs : \rep(f)e=f(\pt)e \textup{ for all }f\in\coeffalg\}.
\end{equation}
Then:
\begin{enumerate}
\item $\rep(\delta^n)\vs_\pt=\vs_{\homeo^n \pt}$  for all $n\in\Z$;
\item If $\{\pt\}$ is an open subset of $\topspace$ and $\chi_\pt\in \coeffalg$ is its characteristic function, then $\rep(\chi_\pt)$ is a projection, and $\rep(\chi_\pt)\vs=\vs_\pt$.
\item If $x_1,\ldots,x_k\in\topspace$ are $k$ different points, then the sum $\sum_{i=1}^k \vs_{x_i}$ is a direct sum.
\end{enumerate}
\end{lemma}

\begin{proof}
If $f\in\coeffalg$ and $e\in \vs_\pt$, then
\[
\rep(f)\rep(\delta)e=\rep(\delta)\rep(\alpha^{-1}(f))e=(\alpha^{-1}(f))(\pt)\rep(\delta)e=f(\homeo \pt)\rep(\delta e).
\]
Hence $\rep(\delta)\vs_\pt \subset \vs_{\homeo \pt}$. Likewise, $\rep(\delta)^{-1} \vs_\pt \subset \vs_{\homeo^{-1}\pt}$; hence $\vs_{\homeo \pt}\subset\rep(\delta)\vs_{\homeo^{-1}\homeo \pt}=\rep(\delta)\vs_\pt$. We conclude that $\rep(\delta)\vs_\pt=\vs_{\homeo \pt}$ for all $\pt\in \topspace$, and this easily implies (1). For (2), suppose that $e\in\rep(\chi_\pt)\vs$. If $f\in\coeffalg$, then
\[
\rep(f)e=\rep(f)\rep(\chi_\pt)e=\rep(f\chi_\pt)e=\rep(f(\pt)\chi_\pt)e=f(\pt)\rep(\chi_e)e=f(\pt)e.
\]
Hence $\rep(\chi_\pt)\vs\subset \vs_\pt$. Conversely, if $e\in \vs_\pt$, then in particular $\rep(\chi_\pt)e=\chi_\pt(\pt)e=e$, so that $\vs_\pt\subset\rep(\chi_\pt)\vs$.
Turning to (3), suppose that $\sum_{i=1}^k e_i=0$, where $e_i\in E_{\pt_i}$  for $i=1,\ldots,k$. For each $i_0$ such that $1\leq i_0\leq k$, there exists $f_{i_0}\in\coeffalg$ such that $f(\pt_{i_0})=1$ and $f(\pt_i)=0$ for $i\neq i_0$. Letting $f_{i_0}$ act shows that $e_{i_0}=0$.

\end{proof}

\subsection{Finite dimensional representations associated with periodic points}\label{subsec:finite_dimensional}

The algebraic equivalence classes of finite dimensional \ai representations of $\lone$ can be described explicitly in terms of the space of finite orbits and $\T$ (see Theorem~\ref{thm:description_finite_dimensional_representations}). We shall now proceed towards this result.

One can associate a \supast-representation with $\pt\in\perpoints$ and a unimodular complex number $\lambda\in\T$, as follows. Let $p$ be the period of $\pt$, and let $\hs_{\pt,\lambda}$ be a Hilbert space with orthonormal basis $\{e_0,\ldots,e_{p-1}\}$. Let $T_\lambda\in\bounded(\hs)$ be represented with respect to this basis by the matrix
$$
\left(
\begin{array}{ccccc}
0 & 0 & \ldots & 0 & \lambda \\
1 & 0 & \ldots & 0 & 0 \\
0 & 1 & \ldots & 0&0\\
\vdots & \vdots & \ddots &\vdots &\vdots \\
0 & 0 & \ldots & 1&0
\end{array}\right).
$$
For $f\in\coeffalg$, let $\rho_{\pt}(f)$ be represented with respect
to this basis by the matrix
$$
\left(
\begin{array}{cccc}
f({\pt}) & 0 & \ldots & 0 \\
0 & f (\homeo {\pt}) & \ldots & 0 \\
\vdots & \vdots & \ddots & \vdots \\
0 & 0 & \ldots & f(\homeo^{p-1} {\pt})
\end{array} \right).
$$
It is easily checked that $\rho$ and $T$ meet the boundedness and covariance requirements in the second part of  Lemma~\ref{lem:how_to_build_representations}; hence there exists a unique \supast-representation $\rep_{\pt,\lambda}:\lone\to\bounded(\hs_{\pt,
\lambda})$ such that $\rep_{\pt,\lambda}\rest{\coeffalg}=\rho_{\pt}$ and $\rep(\delta)=T_\lambda$.

\begin{proposition}[Irreducibility and equivalences]\label{prop:finite_dimensional_given}
Let $\pt\in\perpoints$ and let $\lambda\in\T$. Then the \supast-representation $\rep_{\pt,\lambda}$ on $H_{\pt,\lambda}$ is \ains; the dimension of $H_{\pt,\lambda}$ is the cardinality of the orbit of $\pt$.

If $\pt,y\in\perpoints$ and $\lambda,\mu\in\T$, then the following are equivalent:
\begin{enumerate}
\item $\rep_{\pt,\lambda}$ and $\rep_{y,\mu}$ are algebraically equivalent;
\item $\rep_{\pt,\lambda}$ and $\rep_{y,\mu}$ are topologically equivalent;
\item $\rep_{\pt,\lambda}$ and $\rep_{y,\mu}$ are unitarily equivalent;
\item  The orbits of $\pt$ and $y$ coincide, and $\lambda=\mu$.
\end{enumerate}
\end{proposition}

\begin{proof}
Suppose that $\pt$ has period $p$. If $L\subset\hs_{\pt,\lambda}$ is a non-zero $\lone$-invariant subspace of $H_{\pt,\lambda}$, choose $h=\sum_{i=1}^p \xi_i e_i\in L$ with $\xi_{i_0}\neq 0$ for some $i_0$ such that $1\leq i_0\leq p$. Since the points $\pt,\ldots,\homeo^{p-1}\pt$ are different, one can choose $f\in\coeffalg$ such that $f(\homeo^{i_0}\pt)=1$ and $f$ vanishes at the other points of the orbit of $\pt$. Letting $f$ act on $h$, we see that $e_{i_0}\in L$, and then the action of $\delta$ implies that $L=H$. Hence $\rep_{\pt,\lambda}$ is \ains.

We turn to the equivalences.

Suppose that $\rep_{\pt,\lambda}$ and $\rep_{y,\mu}$ are algebraically equivalent, where $\pt$ has period $p$ and $y$ has period $n$. Since the dimensions of the representation spaces are then equal, we must have $n=p$. Furthermore, since $\left(\rep_{\pt,\lambda}(\delta)\right)^p=\lambda$ and $\left(\rep_{\pt,\mu}(\delta)\right)^n=\mu$, we can then also conclude that $\lambda=\mu$. Since $H_{\pt,\lambda}$ is the direct sum of common eigenspaces of the elements of $\rep(\coeffalg)$, with each summand corresponding to a point of the orbit of $\pt$, this must also be the case for $H_{y,\mu}$. Since there is an analogous decomposition of $H_{y,\mu}$ in terms of the orbit of $y$, the third part of Lemma~\ref{lem:orbits_eigenvalues} implies that the orbit of $\pt$ is included in the orbit of $y$. The converse inclusion follows likewise, and hence the orbits are equal. This shows that (1) implies (4).

In order to show that (4) implies (3), it is sufficient to prove that $\rep_{\pt,\lambda}$ and $\rep_{\homeo \pt,\lambda}$ are unitarily equivalent for $\lambda\in\T$ and $\pt\in\pperpoints$ with $p\geq 2$. For this, let $e_0,\ldots,e_{p-1}$\ be the orthonormal basis of $\hs_{\pt,\lambda}$ as described  in the construction of the representation $\rep_{\pt,\lambda}$ on $\hs_{\pt,\lambda}$, and let $e^\prime_0,\ldots,e^\prime_{p-1}$ be the orthonormal basis of $\hs_{\homeo \pt,\lambda}$ as described in the construction of the representations $\rep_{\homeo \pt,\lambda}$ on $\hs_{\homeo \pt,\lambda}$. Define $U:\hs_{\pt,\lambda}\to\hs_{\homeo \pt,\lambda}$ by $Ue_0=e^\prime_{p-1}$ and by $Ue_j=\lambda e^\prime_{j-1}$  for $j$ such that $1 \leq j\leq p-1$. Then it is easily checked that $U$ implements a unitary equivalence between $\rep_{\pt,\lambda}$ and $\rep_{\homeo \pt,\lambda}$. Hence (4) implies (3).

It is trivial that (3) implies (2), and that (2) implies (1).
\end{proof}

The picture for the finite dimensional \ai representation of $\lone$ is completed by the following result. We use the notation as in \eqref{eq:common_eigenspace} in its proof, that uses the second statement in Theorem~\ref{thm:material_on_algebraically_irreducible_representations}.(2) in an essential way.

\begin{proposition}[Exhaustion]\label{prop:structure_of_finite_dimensional_representations}
Let $\rep:\lone\to\linear(\vs)$ be an \ai representation of $\lone$ on a finite dimensional vector space $\vs$. Then there exist $\pt\in\perpoints$ and $\lambda\in\T$ such that $\rep$ and $\rep_{\pt,\lambda}$ are algebraically equivalent.
\end{proposition}

\begin{proof}
Since $\rep(\coeffalg)$ is a commuting family of linear maps on the finite dimensional complex vector space $\vs$, there exists a common eigenvector $e_0$ for this family. Consequently, there exists a point $\pt\in\topspace$ such that $\vs_\pt\neq 0$. If the points $\pt,\ldots,\homeo^{\dim (\vs)}\pt$ were all different, then the first and third parts of  Lemma~\ref{lem:orbits_eigenvalues} would imply that $\dim (\vs)\geq\dim\left(\sum_{i=0}^{\dim (\vs)}\vs_{\homeo^i\pt}\right)=\dim\left(\bigoplus_{i=0}^{\dim (\vs)}\vs_{\homeo^i\pt}\right)\geq\dim(\vs)+1$. This contradiction implies that $\pt$ is a periodic point. Let $p\geq 1$ be such that $\pt\in\pperpoints$. Then $\rep(\delta)^{p}\vs_\pt=\vs_{\homeo^p \pt}=\vs_\pt$; hence there exist a non-zero $e_0\in\vs_\pt$ and $\lambda\in\C$\ such that $\rep(\delta)^p e_0=\lambda e_0$. Since $\rep$ is automatically unital, $\rep(\delta)$ is invertible, and hence $\lambda\neq 0$. The second part of Theorem~\ref{thm:material_on_algebraically_irreducible_representations} allows us to introduce a norm on $\vs$ such that $\rep$ is a continuous representation, and then Lemma~\ref{lem:how_to_build_representations} implies that $\rep(\delta)$ is two-sided power bounded. In particular, $\Vert\rep(\delta)^{kp}e_0\Vert=|\lambda|^{k}\Vert e_0\Vert$ is bounded as $k$ ranges over $\Z$. Hence $\lambda\in\T$. To conclude the proof, we let $e_j=\rep(\delta)^j e_0\in\vs_{\homeo^j x}$ for $j= 0,\ldots,p-1$. It follows from Lemma~\ref{lem:orbits_eigenvalues} that the $e_j$ are linearly independent. Furthermore, it is easy to see that they span a non-zero subspace of $\vs$ that is invariant under $\rep(\delta)$, $\rep(\delta^{-1})$, and $\rep(\coeffalg)$. Since it is closed in our topology and $\rep$ is continuous, it is in fact invariant under $\lone$. Therefore it equals $\vs$,  and we conclude that the $e_j$ form a basis of $\vs$. It is immediate that, with respect to this basis, the matrix of $\rep(\delta)$ and, for every $f\in\coeffalg$, the matrix of $\rep(f)$ are as in the construction of $\rep_{\pt,\lambda}$ outlined above. The uniqueness statement in the first part of Lemma~\ref{lem:how_to_build_representations} then implies that $\rep$ and $\rep_{\pt,\lambda}$ are algebraically equivalent representations of $\lone$.
\end{proof}

The following description of the algebraic equivalence classes of finite dimensional \ai representations of $\lone$ in terms of an orbit space is now clear from Propositions~\ref{prop:finite_dimensional_given} and~\ref{prop:structure_of_finite_dimensional_representations}.

\begin{theorem}[Finite dimensional representations]\label{thm:description_finite_dimensional_representations} The Banach algebra $\lone$ has finite dimensional  \ai representations if and only if $\perpoints\neq \emptyset$. In that case, let
$\orbspace{\perpoints}$ be the space of finite orbits. If $\orb\in\orbspace{\perpoints}$ and $\lambda\in\T$, choose $\pt\in\orb$, and let $\Xi_{\orb,\lambda}$ denote the algebraic equivalence class of $\rep_{\pt,\lambda}$. Then this is well defined, and $\Xi$ is a bijection between $\orbspace{\perpoints}\times \T$ and the collection of algebraic equivalence classes of finite dimensional \ai representations of $\lone$.

Each of the representation spaces can be supplied with a Hilbert space structure such that the representation is a \supast-representations. In particular, all corresponding primitive ideals are selfadjoint.
\end{theorem}

We conclude by showing that the \ai representations of $\lone$ are all finite dimensional when $\topspace$ is a finite set (see Theorem~\ref{thm:finite_set}). This will be used in the proof of Corollary~\ref{cor:properties_induced_system}, which, in turn, is necessary to establish Proposition~\ref{prop:all_finite_dimensional}. The latter is a more general result than Theorem~\ref{thm:finite_set} on which it partly builds. It is based on the following result, valid for general $\topspace$.

\begin{proposition}\label{prop:must_be_finite_dimensional}
Let $\rep:\lone\to \linear(E)$ be an \ai representation of $\lone$ on the vector space $\vs$. Suppose that $\pt\in\perpoints$ is such that $\{\pt\}$ is open, and that $\rep(\chi_\pt)\neq 0$, where $\chi_\pt\in\coeffalg$ denotes the characteristic function of $\{\pt\}$. Then there exists $\lambda\in\T$ such that $\rep$ is algebraically equivalent to the \ai representation $\rep_{\pt,\lambda}$ associated with the periodic point $\pt$ and $\lambda\in\T$. In particular, $\vs$ is finite dimensional.
\end{proposition}

\begin{proof}
By the second part of Theorem~\ref{thm:material_on_algebraically_irreducible_representations}, we may assume that $\vs$ is a normed space and that $\rep$ is a continuous representation. Let $p$ be the period of $\pt$, and let $\chi\in\coeffalg$ denote the characteristic function of the orbit of $\pt$. It follows from $\chi=\sum_{j=0}^{p-1} \chi_{\homeo^j \pt}$ and the second and third parts of Lemma~\ref{lem:orbits_eigenvalues} that $\rep(\chi)\vs=\bigoplus_{j=0}^{p-1}\vs_{\homeo^j \pt}$. Then clearly $\rep(\chi)\vs$ is invariant under $\coeffalg$, and the first part of Lemma~\ref{lem:orbits_eigenvalues} implies that it is invariant under $\delta$ and $\delta^{-1}$. Since $\rep(\chi)$ is a continuous projection, its range is closed, and we can now conclude from the continuity of $\rep$ that $\rep(\chi)\vs$ is invariant under $\rep(\lone)$.
Since $\rep(\chi_\pt)\neq 0$, we have $\rep(\chi)\vs\neq\{0\}$, and hence $\vs=\rep(\chi)\vs=\bigoplus_{j=0}^{p-1}\vs_{\homeo^j \pt}$.

We note that $\rep(\delta)^p$ leaves each $\vs_{\homeo^j \pt}$ invariant, as a consequence of the first part of Lemma~\ref{lem:orbits_eigenvalues}. This implies that $\rep(\delta)^p$ commutes with $\rep(\coeffalg)$. Since it obviously commutes with $\rep(\delta)$ and $\rep(\delta)^{-1}$, it commutes with $\rep(\lone)$ by continuity. Hence the first part of Theorem~\ref{thm:material_on_algebraically_irreducible_representations} implies that $\rep(\delta)^p=\lambda$ for some $\lambda\in\C$. Since $\rep(\delta)^p$ is double-sided power bounded, we must have $\lambda\in\T$. If we take $e\in\vs_\pt$ to be non-zero, then, by Lemma~\ref{lem:orbits_eigenvalues}, $\{e,\rep(\delta)e,\ldots,\rep(\delta)^{p-1}e\}$ is independent. Its closed linear span is clearly invariant under $\rep(\coeffalg)$,  $\rep(\delta)$, and $\rep(\delta)^{-1}$; by the continuity of $\rep$ it is then invariant under $\rep(\lone)$. Therefore it equals $\vs$. It is now clear from the uniqueness statement in the first part of Lemma~\ref{lem:how_to_build_representations} that $\rep$ and $\rep_{\pt,\lambda}$ are algebraically equivalent.
\end{proof}

Now assume that $\topspace$ is finite and that $\pi$ is an \ai representation of $\lone$ on a vector space $\vs$. Since $1=\sum_{\pt\in\topspace}\chi_x$, and an \ai representation is non-zero and automatically unital, there exists $\pt\in\topspace$ such that $\rep(\chi_x)\neq 0$. Hence Proposition~\ref{prop:must_be_finite_dimensional} applies. In conclusion, we have the following, as a precursor to and stepping stone for Proposition~\ref{prop:all_finite_dimensional}.

\begin{theorem}\label{thm:finite_set}
Suppose that $\topspace$ is a finite set. Then all \ai representations of $\lone$ are finite dimensional, and the collection of algebraic equivalence classes of \ai representations of $\lone$ can be identified with $\orbspace{\topspace}\times \T$ as in Theorem~\ref{thm:description_finite_dimensional_representations}. Each of the representation spaces can be supplied with a Hilbert space structure such that the representation is a \supast-representations. In particular, all primitive ideals of $\lone$ are selfadjoint.
\end{theorem}

\subsection{Infinite dimensional representations associated with aperiodic points}\label{subsec:infinite_dimensional}
In this section, starting from an aperiodic point, we shall define representations of $\lone$ on the two-sided $\ell^p(\Z)$-spaces for all $p\in[1,\infty]$.  We shall determine the algebraic, topological and isometrical equivalences (see Propositions~\ref{prop:equivalences} and~\ref{prop:top_equivalences}), and decide the algebraic and topological irreducibility (see Theorem~\ref{thm:irreducibility_infinite_dimensional}). The kernels of these representations are closed ideals that turn out to be selfadjoint.

The representations on $\ell^1(\Z)$ are \ains. In contrast to the finite dimensional case, we do not know all infinite dimensional \ai representations up to algebraic equivalence. Later, in Theorem~\ref{thm:exhaustive}, we shall see that, if $\topspace$ is metrizable, we nevertheless know all primitive ideals corresponding to such representations.

To construct these representations on $\ell^p(\Z)$ for $\pt\in\aperpoints$, we first let $p\in[1,\infty)$ and, for $k\in\Z$, we let $e_k$ denote the element of $\ell^p(\Z)$ with 1 in the $k$th coordinate and zero elsewhere. Let $S\in\bounded(\ell^p(\Z))$ be the right shift, determined by $Se_k=e_{k+1}$ for $k\in\Z$. For $f\in\coeffalg$, let $\rep^p_\pt(f)\in\bounded(\ell^p(\Z))$ be determined by $ \rep^p_\pt(f)e_k= f(\homeo^k \pt )e_k$  for all $k\in\Z$. One can then easily see that $\rep^p_\pt$ and $S$ satisfy the requirements of Lemma~\ref{lem:how_to_build_representations}. Hence there exists a unique unital continuous representation $\rep^p_\pt:\lone\to\bounded(\ell^p(\Z))$ such that
\begin{equation}\label{eq:aperiodic_point_representation}
\rep^p_\pt\left(\sumn f_n\delta^n\right)=\sumn\rep^p_\pt(f)S^n
\end{equation}
for $\loneelement{n}\in\lone$.
Furthermore, $\rep^p_\pt$ is contractive and, for $p=2$, it is a unital \supast-representation on the Hilbert space $\ell^2(\Z)$.
For $p=\infty$, the construction is similar. We let $S$ denote the right shift on $\ell^\infty(\Z)$, and, for $f\in\coeffalg$, we let the operator $\rep^\infty_\pt(f)$ act on a two-sided sequence via multiplication with $f(\homeo^k \pt)$ in the $k$th coordinate  for all $k\in\Z$.  Again Lemma~\ref{lem:how_to_build_representations} implies that there exists a unique unital continuous representation $\rep^\infty_\pt:\lone\to\bounded(\ell^p(\Z))$
such that \eqref{eq:aperiodic_point_representation} holds for $p=\infty$. Then $\rep_\pt^\infty$ is contractive.

The question concerning equivalence of these representations for fixed $p$ is easily answered.

\begin{proposition}[Equivalences]\label{prop:equivalences}
Let $\pt,y\in\aperpoints$ and let $p\in[1,\infty]$. Then the following are equivalent:
\begin{enumerate}
\item $\rep^p_\pt$ and $\rep^p_y$ are algebraically equivalent;
\item $\rep^p_\pt$ and $\rep^p_y$ are topologically equivalent;
\item $\rep^p_\pt$ and $\rep^p_y$ are isometrically equivalent;
\item The orbits of $\pt$ and $y$ coincide.
\end{enumerate}
\end{proposition}

\begin{proof}
We prove that (1) implies (4). If $\rep:\coeffalg\to\linear(\vs)$ is a representation of $\coeffalg$ on a vector space $\vs$, then a common eigenspace for the action of $\coeffalg$ on $\vs$ is equal to $\vs_z=\{e\in \vs : \rep(f)e=f(z)e\}$ for a uniquely determined $z\in \topspace$. It is routine to check that, for the representation $\rep_x^p$, these subspaces of $\ell^p(\Z)$ are non-zero\textemdash and in fact one dimensional and spanned by a standard vector $e_k$\textemdash if and only if $z$ is in the orbit of $\pt$. Therefore, if $\rep_{\pt,\lambda}$ and $\rep_{y,\mu}$ are algebraically equivalent, then the orbits of $\pt$ and $y$ must coincide.

If (4) holds, say $y=\homeo^j \pt$ for some unique $j\in\Z$, then $S^{-j}$ is an isometric equivalence between $\rep^p_\pt$ and $\rep^p_y$. Hence (4) implies (3).

It is trivial that (3) implies (2), and that (2) implies (1).
\end{proof}

\begin{remark}\label{rem:schur}Let $p\in[1,\infty]$. Using the one dimensionality of the common eigenspaces for the action of $\coeffalg$, and considering the action of $\delta$ and $\delta^{-1}$, one sees easily that, for any $T\in\linear(\ell^p(\Z))$ commuting with the action of $\lone$, there exists $\lambda\in\C$ such that $Te_k=\lambda e_k$ for all $k\in\Z$. If $p\in[1,\infty)$, this implies that the bounded intertwining operators for $\rep_\pt^p$ are the multiples of the identity.
\end{remark}

It is natural to ask, in addition, which algebraic equivalences exist between $\rep_\pt^p$ and $\rep_{y}^r$ if $p$ and $r$ are such that $1\leq p <r\leq\infty$. As in the proof of Proposition~\ref{prop:equivalences}, the orbits of $\pt$ and $y$ must then coincide. Furthermore, as will become obvious from Theorem~\ref{thm:irreducibility_infinite_dimensional}, we must then have $1<p<r\leq\infty$, but additional information is not available at this moment. For topological equivalence, however, the picture is clear.

\begin{proposition}\label{prop:top_equivalences}
Let $\pt,y\in\aperpoints$ and let $p,r\in[1,\infty]$. Then the following are equivalent:
\begin{enumerate}
\item $\rep^p_\pt$ and $\rep^r_y$ are topologically equivalent;
\item $\rep^p_\pt$ and $\rep^r_y$ are isometrically equivalent;
\item The orbits of $\pt$ and $y$ coincide, and $p=r$.
\end{enumerate}
\end{proposition}

\begin{proof}
If (1) holds, then $\rep^p_\pt$ and $\rep^r_y$ are algebraically equivalent and, as remarked preceding the proposition, the orbits of $\pt$ and $y$ must then coincide. Also, if (1) holds, then $\ell^p(\Z)$ and $\ell^r(\Z)$ are topologically isomorphic. Since each of these spaces is clearly topologically isomorphic to its one-sided version, it then follows from \cite[Corollary~2.1.6]{AK} and the non-reflexivity and non-separability of $\ell^\infty(\N)$ that we must have $p=r$. Hence (1) implies (3). From Proposition~\ref{prop:equivalences}, we see that (3) implies (2), and trivially (2) implies (1).
\end{proof}

We shall now investigate the algebraic and topological irreducibility of the representations $\rep^p_\pt$ for $\pt\in\aperpoints$ and $p\in[1,\infty]$. Theorem~\ref{thm:irreducibility_infinite_dimensional} gives a complete answer.

The hardest part is the algebraic irreducibility of $\rep^1_\pt$, which we now take up. For the proof we need two lemmas, where the second builds on the first.

\begin{lemma}\label{lem:first_preparation}
Let $\pt\in\aperpoints$, and suppose that $\rho=\sumn \lambda_n e_n\in\ell^1(\Z)$ with $\lambda_0=1$. If $\tau\in\ell^1(\Z)$
and $\varepsilon>0$ are given, then there exists $a\in\lone$ such that  $\left\Vert
\tau - \rep^1_\pt  (a)\rho\right\Vert < \varepsilon$ and $\Vert a\Vert \leq
\Vert \tau \Vert$.
\end {lemma}

\begin{proof}
Let $\tau = \sumn \mu_n e_n$. For every integer $N\geq 0$, let $\rho_N
= \sum_{\vert n \vert \leq N}\lambda_n e_n$ and let $\tau _N = \sum_{\vert
n \vert \leq N}\mu_n e_n$. If $\varepsilon>0$ is given, choose $N_1,N_2\geq
0$ such that $\Vert \tau \Vert\Vert {\rho - \rho_{N_1}}\Vert + \Vert {\tau
- \tau_{N_2}}\Vert   < \varepsilon $. Next, choose $f\in\coeffalg$ such that
$\Vert f\Vert=1$, $f(\pt ) = 1$, and $f(\homeo^j \pt ) = 0$ for all $j$ such that $0<\vert j
\vert \leq N_1$. Since the points $\homeo^j \pt$ for $j=0,\ldots,N_1$ are all
different, this is indeed possible. We then have
$$
\rep^1_\pt(f)\rho_{N_1} = \sum_{\vert n \vert\leq N_1}\lambda_n f(\homeo^n
\pt )e_n = e_0.
$$
Let
$$
a = \left(\sum_{\vert n \vert \leq N_2}\mu_n \delta^n\right)\cdot f;
$$
then
$$
\rep^1_\pt(a)\rho_{N_1}= \rep^1_\pt\left(\left(\sum_{\vert n \vert \leq N_2 }\mu
_n \delta^n\right)\cdot f\right)\rho _{N_1} = \rep^1_\pt\left(\sum_{\vert n
\vert \leq N_2}\mu_n\delta^n\right) e_0 = \tau_{N_2}.
$$
Furthermore,
$$
\Vert a \Vert \leq \left\Vert \sum_{\vert n \vert \leq N_2}\mu_n \delta^n
\right\Vert\Vert f\Vert = \sum _{\vert n \vert \leq N_2}\vert \mu_n \vert
\leq \Vert \tau \Vert,
$$
and
\begin{align*}
\Vert \rep^1_\pt(a)\rho - \tau \Vert &\leq \Vert \rep^1_\pt(a)(\rho - \rho_{N_1})\Vert
+ \Vert \rep^1_\pt(a) \rho _{N_1} - \tau_{N_2}\Vert + \Vert \tau_{N_2} - \tau
\Vert
\\
& \leq \Vert a\Vert \Vert \rho - \rho_{N_1}\Vert + \Vert \tau_{N_2} - \tau
\Vert \leq \Vert \tau \Vert \Vert \rho - \rho_{N_1}\Vert + \Vert \tau _{N_2}
- \tau \Vert
\\
& < \epsilon.
\end{align*}
Hence $a$ meets all requirements.
\end{proof}

\begin{lemma}\label{lem:second_preparation}
Let $\pt\in\aperpoints$, and suppose that $\rho=\sumn \lambda_n e_n\in\ell^1(\Z)$ with $\lambda_0=1$. If $\tau\in\ell^1(\Z)$
and $\varepsilon>0$ are given, then there exists $a\in\lone$ such that $\rep^1_\pt(a)\rho
= \tau$ and $\Vert a\Vert \leq (1 + \varepsilon) \Vert \tau \Vert$.
\end{lemma}

\begin{proof}
We may assume that $\tau\neq 0$. Fix $\gamma$ such that $0 < \gamma < 1$. We claim that there exist
$a_1,a_2,\ldots\in\lone$ such that
\begin{equation}\label{eq:approximation}
\left \Vert \tau - \rep^1_\pt \left( \sum _{j=1}^N a_j\right)\rho \right
\Vert < \gamma^N \Vert \tau \Vert
\end{equation}
for all $N\geq 1$ and $\Vert a_j\Vert\leq\gamma^{j-1}\Vert\tau\Vert$ for
all $j\geq 1$.
 To see this, we use an inductive construction. First apply Lemma~\ref{lem:first_preparation}
with $\varepsilon=\gamma\Vert\tau\Vert>0$ to find $a_1\in\lone $ such that
$$
\Vert \tau - \rep^1_\pt(a_1)\rho \Vert < \gamma \Vert \tau \Vert
$$
and $\Vert a_1 \Vert \leq \Vert \tau \Vert$. Let $N\geq 1$, and assume that
$a_1,a_2,\ldots,a_N\in\lone$ have already been chosen such that
$$
\left \Vert \tau - \rep^1_\pt \left( \sum _{j=1}^n a_j\right)\rho \right
\Vert < \gamma^n \Vert \tau \Vert
$$
for all $n$ such that $1\leq n\leq N$, and $\Vert a_j\Vert\leq\gamma^{j-1}\Vert\tau\Vert$
for all $j$ such that $1\leq j\leq N$.
We apply Lemma~\ref{lem:first_preparation} again with $\varepsilon=\gamma^{N+1}\Vert\tau\Vert>0$,
and find $a_{N + 1}\in\lone$ such that
$$
\left\Vert \left(\tau - \rep_\pt^1\left(\sum_{j=1}^N a_j\right)\rho\right)  -
\rep^1_\pt(a_{N + 1})\rho \right\Vert < \gamma^{N + 1} \Vert \tau \Vert
$$
and
$$
\Vert a_{N + 1} \Vert \leq \left\Vert \tau - \rep^1_\pt \left(\sum_{j=1}^N
a_j \right)\rho \right\Vert < \gamma^N \Vert \tau \Vert.
$$
This completes the inductive step in the construction and establishes our
claim.

If we let $ a = \sum_{j=1}^{\infty} a_j\in\lone $, then
$$
\Vert a \Vert \leq \sum_{j=1}^{\infty} \gamma^{j - 1} \Vert \tau \Vert =
\frac{1}{1 - \gamma} \Vert \tau \Vert.
$$
Furthermore, if we let $N\to\infty$ in \eqref{eq:approximation}, we see that
$$
\rep^1_\pt(a) \rho = \tau.
$$
Since this can be done for any $\gamma$ such that $0<\gamma<1$, the proof is complete.
\end{proof}

\begin{proposition}\label{prop:one_algebraically_irreducible}
Let $\pt\in\aperpoints$. Then the representation $\rep^1_\pt$ of $\lone$ on $\ell^1(\Z)$ is \ains. In fact, for any $\rho = \sumn \lambda_n e_n\in\ell^1(\Z)$ such that $\rho\neq 0$, $\tau\in\ell^1(\Z)$, and $\varepsilon>0$, there exists $a\in\lone$ such that $\rep_\pt(a)\rho = \tau$ and
$$
\Vert a \Vert \leq \frac{(1+\varepsilon)}{\max_{n} \vert \lambda_n \vert}\Vert
\tau \Vert.
$$
\end{proposition}

\begin{proof}
Suppose that $\vert \lambda_{n_0} \vert = \max_n \vert \lambda_n \vert\neq 0$.
Then $\rep^1_\pt(\lambda_{n_0}^{-1} \delta ^{- n_0})\rho $ satisfies the condition
for $\rho$ in Lemma~\ref{lem:second_preparation}. Hence there exists $a'\in\lone$
such that
$$
\rep^1_\pt(a')\left(\rep^1_\pt (\lambda_{n_0}^{-1} \delta^{- n_0})\right) \rho
= \tau
$$
and
$$
\Vert a' \Vert \leq ( 1 + \varepsilon) \Vert \tau \Vert.
$$
Then $a = a'\cdot\lambda_{n_0}^{-1} \delta^{- n_{0}}$ has the required properties.
\end{proof}

The algebraic irreducibility for other cases is easily dealt with.

\begin{lemma}\label{lem:p_algebraically_reducible}
Let $\pt\in\aperpoints$ and let $p\in(1,\infty]$. Then the representation $\rep^p_\pt$ of $\lone$ on $\ell^p(\Z)$ is not algebraically irreducible. In fact, $\rep^p_\pt(\lone )e_0$ is a proper invariant subspace.
\end{lemma}

\begin{proof}
A moment's thought shows that $\rep^p_\pt(\lone)e_0=\ell^1(\Z)$, viewed canonically as a subspace of $\ell^p(\Z)$. For $p\in(1,\infty]$ this is indeed a proper subspace.
\end{proof}

Topological irreducibility is settled as follows.

\begin{proposition}\label{prop:topological_irreducibility}
Let $\pt\in\aperpoints$. For $p\in[1,\infty)$, the representation $\rep^p_\pt$ of $\lone$ on $\ell^p(\Z)$ is topologically irreducible. For $p=\infty$, it is not topologically irreducible. In fact,  $\overline{\rep^\infty_\pt(\lone)e_0}$ is a proper closed invariant subspace of $\ell^\infty(\Z)$.
\end{proposition}

\begin{proof}
Assume that $p\in[1,\infty)$ and that $L$ is a non-zero closed invariant subspace of $\ell^p(\Z)$. Let $\rho=\sumn \lambda_n e_n\in L$ be non-zero. After applying $\rep^p_\pt(\delta)$ or its inverse a number of times, followed by scaling, we may assume that $\lambda_0=1$. Let $\varepsilon>0$, and choose $N$ such that $\sum_{|n|>N}|\lambda_n|^p<\varepsilon^p$. Since the $2N+1$ points $\homeo^j \pt$ for $|j|\leq N$ are all different, we can choose $f\in\coeffalg$ such that $\Vert f\Vert=1$, $f(\pt)=1$, and $f(\homeo^j \pt)=0$ for all $j$ such that $1\leq |j|\leq N$. Then
\begin{align*}
\left\Vert \rep^p_\pt(f)\rho-e_0\right\Vert&=\left\Vert\sum_{|n|>N}f(\homeo^n \pt)\lambda_n e_n\right\Vert\\
&=\left(\sum_{|n|>N}\left|f(\homeo^n \pt)\lambda_n\right|^p\right)^{1/p}\\
&\leq \left(\sum_{|n|>N}\left|\lambda_n\right|^p\right)^{1/p}\\
&<\varepsilon.
\end{align*}
Since $L$ is closed, this implies that $e_0\in L$. But then also $\rep^p_\pt(\lone)e_0\subset L$. As already observed earlier, $\rep^p_\pt(\lone)e_0$ equals the canonical copy of $\ell^1(\Z)$ in $\ell^p(\Z)$. Since this copy is clearly dense in $\ell^p(\Z)$ for $p\in[1,\infty)$, we must have $L=\ell^p(\Z)$.

If $p=\infty$, then $\rep^\infty_\pt(\lone)e_0$ is not dense in $\ell^\infty(\Z)$. Indeed, this is the image of $\ell^1(\Z)$ under the canonical inclusion map; since this inclusion is continuous and $\ell^1(\Z)$ is separable, the density of the image of $\ell^1(\Z)$ would then imply that $\ell^\infty(\Z)$ is separable.
\end{proof}

We summarise our results on irreducibility in the following theorem, including Remark~\ref{rem:schur}.

\begin{theorem}[Irreducibility]\label{thm:irreducibility_infinite_dimensional}
Let $\pt\in\aperpoints$ and let $p\in[1,\infty]$. Then the unital contractive representation $\rep^p_\pt:\lone\to\bounded(\ell^p(\Z))$ is:
\begin{enumerate}
\item \Ai if $p=1$;
\item Not algebraically irreducible, but still topologically irreducible, if $p\in(1,\infty)$;
\item Not topologically irreducible if $p=\infty$.
\end{enumerate}
If $p\in[1,\infty)$, then the commutant of $\rep^p_\pt(\lone)$ in $\bounded(\ell^p(\Z))$ equals the multiples of the identity operator.
\end{theorem}

\begin{remark}\label{rem:not_on_Hilbert_space}\quad
\begin{enumerate}
\item  In the algebraically irreducible case $p=1$, the statement on bounded intertwining operators is also clear from the first part of Theorem~\ref{thm:material_on_algebraically_irreducible_representations}, since then even the commutant of $\rep^1_\pt(\lone)$ in $\linear(\ell^1(\Z))$ consists of the scalars. For $p=2$ it follows from topological irreducibility and Schur's Lemma, but for $p$ such that $1<p<2$ or $2<p<\infty$ there do not seem to be alternative general approaches.
\item As a consequence of (2) and (3) in Theorem~\ref{thm:material_on_algebraically_irreducible_representations}, every \ai representation of a Banach algebra comes naturally with the topological isomorphism class of all Banach spaces in which the representation can be realised as a normed representation. Therefore, even though the corresponding primitive ideal is selfadjoint according to Lemma~\ref{lem:infinite_orbit_ideal} below, the \ai representation $\rep^1_\pt$  for $\pt\in\aperpoints$ is not algebraically equivalent to a normed representation on a Hilbert space\textemdash let alone to a \supast-representation on such a space\textemdash since a reflexive space is not topologically isomorphic to $\ell^1(\Z)$.
\end{enumerate}
\end{remark}

\begin{remark}
The representations as associated with aperiodic points in this section are members of a more general class of (usually) infinite dimensional representations of $\lone$, namely the representations associated with invariant measures. If $\mu$ is a $\homeo$-invariant Borel measure on $\topspace$, then, for $p\in[1,\infty]$, there is a natural representation $\pi^p_\mu$ of $\lone$ on $L^p(\topspace,\mu)$, in which $\coeffalg$ acts as multiplication operators, and in which $(\pi_\mu^p(\delta)f)(x)=f(\homeo^{-1}x)$ for $f\in L^p(\topspace,\mu)$ and $\pt\in\topspace$. If $\mu$ is the counting measure corresponding to the orbit of an aperiodic point $x\in\topspace$, then, for $p\in[1,\infty]$, the representation $\pi_\mu^p$ is isometrically equivalent to the representation $\pi_x^p$. Quite in contrast to Theorem~\ref{thm:irreducibility_infinite_dimensional}, it is presently unknown under which necessary and sufficient conditions these general representations $\rep_\mu^p$ are topologically or algebraically irreducible. If $\mu$ is finite and $\pi_\mu^p$ is topologically irreducible, then $\mu$ must be ergodic, but the converse is still open. The representations $\pi_x^1$ in the present section are the only presently known examples of algebraically irreducible representations of $\lone$.

For metrizable $\topspace$, a thorough study of topologically irreducible \supast-representations of $\lone$ on Hilbert spaces and their algebraic (ir)reducibility is made in \cite{KishiTomi}, where these representations are related to ergodic $\homeo$-quasi-invariant probability measures on $\topspace$. Here is a sample result: If $\mu$ is an ergodic $\sigma$-invariant non-atomic probability measure on $\topspace$, then $\pi_{\mu}^2$ is a \supast-representation of $\lone$ on $L^2(\topspace,\mu)$ that is topologically irreducible, but not algebraically irreducible. The latter statement is a special case of \cite[Theorem~3.4]{KishiTomi}. We refer the reader to \cite{KishiTomi} for further material in this direction.
\end{remark}

\section{Algebraically irreducible representations and primitive ideals}\label{sec:primitive_ideals}

Section~\ref{sec:representations} provides a basic stockpile of \ai representations and corresponding primitive ideals, all related to the points of $\topspace$. In this section, we combine this information with the technique of induced dynamical systems from Section~\ref{sec:preliminaries}, and obtain various results on general \ai representations and primitive ideals of $\lone$. As a side result, we show that $\lone$ is a Hermitian Banach algebra if all points are periodic. We also have a closer look at the primitive ideals associated with the \ai representations from Section~\ref{sec:representations}, and show how they can be parameterized (see Proposition~\ref{prop:psi_proposition}). Using these primitive ideals, it already follows that $\lone$ is semisimple.

As is to be expected, the primitive ideals are more accessible than the in general more numerous (algebraic equivalence classes of) \ai representations. The fact that, even though we have no explicit description, all primitive ideals of $\lone$ can be shown to be selfadjoint (see Theorem~\ref{thm:primitive_ideals_are_selfadjoint}) is an example of this. Theorem~\ref{thm:exhaustive} is an even clearer illustration: if $\topspace$ is metrizable, then we know all primitive ideals\textemdash they all originate from Section~\ref{sec:representations}\textemdash even though we do not know all \ai representations.

In some results, the enveloping $C^*$-algebra $\cstar$ of $\lone$ makes an appearance. The results for that algebra are not used to obtain results for $\lone$. They are simply cited, and are included to show that some properties are in fact preserved when making the non-trivial passage from $\lone$ to $\cstar$.

Our first main task is to show that all primitive ideals are selfadjoint. For this, we shall use the properties of induced dynamical systems as in Proposition~\ref{prop:induced_system}, Remark~\ref{rem:can_be_assigned}, Proposition~\ref{prop:structure_of_finite_dimensional_representations}, and the notation therein. We start with a lemma. Recall from Section~\ref{sec:preliminaries} that a topological dynamical system $\dynsysshort=\dynsys$ is said to be topologically transitive if $\bigcupn \homeo^n(V)$ is dense in $\topspace$ for every non-empty open subset $V$ of $\topspace$.

\begin{lemma}\label{lem:property_of_induced_system}
Let $\rep:\lone\to\bounded(\vs)$ be a topologically irreducible continuous representation of $\lone$ on a normed space $\ns$. Then the induced dynamical system $\dynsysshort_{\rep}$ is topologically transitive.
\end{lemma}

\begin{proof}
The short proof of \cite[Proposition~4.4]{Tnotes1} in a Hilbert context also works in general, and we include it for the convenience of the reader. Suppose to the contrary that there exists a non-empty open subset $U$ such that $S:=\overline{\bigcup_{n\in\Z}\inducedhomeo^n(U)}\neq\inducedset$, so that $S$ is a proper closed invariant subset of $\inducedset$. Hence, if we let $\textup{ker}(S)=\{f\in\inducedcoeffalg : f\rest{S}=0\}$, then  $\textup{ker}(S)\neq \{0\}$. Since $\rep^\prime$ is injective by the third part of Proposition~\ref{prop:induced_system}, we have $\rep^\prime(\textup{ker}(S))E\neq \{0\}$. The invariance of $S$ implies that $\rep^\prime(\textup{ker}(S))E$ is not only invariant under $\rep^\prime(C(\inducedset))$, but also under $\rep^\prime(\induceddelta)$ and $\rep^\prime(\induceddelta^{-1})$. Since $\rep^\prime$ is continuous, the non-zero closed subspace $\overline{\rep^\prime(\textup{ker}(S))E}$ is invariant under $\rep^\prime(\ell^1(\inducedsystemshort))$. We conclude that $\overline{\rep^\prime(\textup{ker}(S))E}=E$. Now take a non-zero $f\in C(\inducedset)$ such that $\supp(f)\subset U$. Then $f\,\textup{ker}(S)=\{0\}$ since $U\subset S$, and this implies that $\rep^\prime(f)=0$. But this contradicts the injectivity of $\rep^\prime$ on $C(\inducedset)$.
\end{proof}

Before stating a consequence of the previous result, we recall from Section~\ref{sec:preliminaries}  that  a topological dynamical system $\dynsysshort=\dynsys$ is said to be topologically free if the subset $\aperpoints$ of aperiodic points of $\homeo$ is dense in $\topspace$.

\begin{corollary}\label{cor:properties_induced_system}
Let $\rep:\lone\to\linear(\vs)$ be an \ai representation of $\lone$ on a vector space $\vs$. Then the induced dynamical system $\dynsysshort_{\rep}$ is topologically transitive. If $\vs$ is infinite dimensional, then $\inducedset$ is an infinite set, and $\dynsysshort_{\rep}$ is topologically free.
\end{corollary}

\begin{proof}
For the first part, we note that, by the second part of Theorem~\ref{thm:material_on_algebraically_irreducible_representations}, we can assume that $\vs$ is normed and that $\rep$ is continuous. Then $\rep$ is certainly topologically irreducible; hence the first statement follows from Lemma~\ref{lem:property_of_induced_system}.
For the second part, we know from the fourth part of Proposition~\ref{prop:induced_system} that the induced representation $\rep^\prime$ of $\inducedlone$ on $\vs$ is algebraically irreducible. If $\inducedset$ were finite, then Theorem~\ref{thm:finite_set} would imply that $\vs$ is finite dimensional. This is not the case, so $\inducedset$ is infinite. Since the proof of \cite[Proposition~3.4]{T4} contains a proof of the statement that a topologically transitive dynamical system on an infinite compact Hausdorff space is topologically free\textemdash a fact that can also be deduced from \cite[Corollary~A.2]{JTStudia}\textemdash the proof is complete.
\end{proof}

We can now reach our desired conclusion.

\begin{theorem}\label{thm:primitive_ideals_are_selfadjoint}Every primitive ideal of $\lone$ is selfadjoint.
\end{theorem}

\begin{proof}
Let $\rep$ be an \ai representation of $\lone$. If $\rep$ is finite dimensional, then $\Ker(\rep)$ is selfadjoint by Theorem~\ref{thm:description_finite_dimensional_representations}. If $\rep$ is infinite dimensional, then Corollary~\ref{cor:properties_induced_system} shows that the induced system $\inducedsystemshort$ is topologically free. Using once more that, by the second part of Theorem~\ref{thm:material_on_algebraically_irreducible_representations}, we can assume that $\vs$ is normed and that $\rep$ is continuous, Proposition~\ref{prop:self_adjoint} then implies that $\Ker(\rep)$ is selfadjoint.
\end{proof}

With Theorem~\ref{thm:primitive_ideals_are_selfadjoint} available, we can now establish the following result on spectral synthesis. In order not to interrupt the line of argumentation, we shall already use the fact that $\lone$ is semisimple (see Theorem~\ref{thm:semisimple}). That result will follow by direct computation on the \ai  representations of $\lone$, as constructed in Section~\ref{sec:representations}, and does not depend on any other results in the present section.

\begin{theorem}\label{thm:all_ideals_selfadjoint}
The following are equivalent:
\begin{enumerate}
\item Every closed ideal of $\lone$ is the intersection of primitive ideals;
\item $\dynsysshort$ is free;
\item  Every closed ideal of $\lone$ is selfadjoint.
\end{enumerate}
\end{theorem}

\begin{proof}
The equivalence of (2) and (3) is the statement of \cite[Theorem~4.4]{DST}. Since we know from Theorem~\ref{thm:primitive_ideals_are_selfadjoint} that all primitive ideals of $\lone$ are selfadjoint, (1) implies (3). We show that (2) implies (1). Let $I$ be a closed ideal of $\lone$; we may assume that $I$ is proper. Consider the non-zero continuous quotient homomorphism $Q:\lone\to\lone/I$. If $\Sigma_Q$ is its induced dynamical system, then $\Sigma_Q$ is free, since even $\Sigma$ as a whole is free. Hence $\Sigma_Q$ is certainly topologically free, and then Proposition~\ref{prop:self_adjoint} shows that $I=\Ker(Q)=\Ker(\inducedrest)$, where $\inducedrest:\lone\to\inducedlone$ is the surjective homomorphism in Proposition~\ref{prop:induced_system}. Since $\inducedlone$ is semisimple by Theorem~\ref{thm:semisimple}, there is a collection of \ai representations of $\inducedlone$ that separates the points of $\inducedlone$. Pulling these back to $\lone$ via the surjective homomorphism $\inducedrest$, one obtains a collection of \ai representations of $\lone$ such that the intersection of the corresponding primitive ideals equals $\Ker(\inducedrest)$, i.e.\  such that this intersection equals $I$.
\end{proof}

Next, we have the following two exhaustion results at the level of \ai representations. The first one was already announced preceding Proposition~\ref{prop:must_be_finite_dimensional} and Theorem~\ref{thm:finite_set}.

\begin{proposition}\label{prop:all_finite_dimensional} The following are equivalent:
\begin{enumerate}
\item Every \ai representation of $\lone$ on a vector space is finite dimensional;
\item $\topspace=\perpoints$;
\item Every \ai representation of $C^{\ast}(\dynsysshort)$ is finite dimensional.
\end{enumerate}
\end{proposition}

\begin{proof}
If $\pt\in\aperpoints$, then the infinite dimensional \ai representations $\rep_{\pt}^1$ in Section~\ref{sec:representations} exist. This shows that (1) implies (2). If $\rep$ is an infinite dimensional \ai representation of $\lone$, then Corollary~\ref{cor:properties_induced_system} shows that its induced system $\inducedsystemshort$ is topologically free. In particular, there is at least one aperiodic point in $\inducedset$, and we see that (2) implies (1).

The equivalence of (2) and (3) follows from \cite[Theorem~4.6.(1)]{Tnotes1}, taken together with Theorem~\ref{thm:cstar_reps}.
\end{proof}

Its counterpart is the following. It is more elaborate than Proposition~\ref{prop:all_finite_dimensional}, because for finite dimensional spaces the notions of algebraic irreducibility and topological irreducibility coincide.

\begin{proposition}\label{prop:all_infinite_dimensional} The following are equivalent:
\begin{enumerate}
\item Every \ai representation of $\lone$ on a vector space is infinite dimensional;
\item Every topologically irreducible continuous representation of $\lone$ on a Banach space is infinite dimensional;
\item Every topologically irreducible representation of $\lone$ on a normed space is infinite dimensional;
\item $\topspace=\aperpoints$;
\item Every \ai representation of $C^{\ast}(\dynsysshort)$ is infinite dimensional.
\end{enumerate}
\end{proposition}

\begin{proof}
If $\pt\in\perpoints$, then the finite dimensional \ai representations $\rep_{\pt,\lambda}$ for $\lambda\in\T$ from Section~\ref{sec:representations} exist. This shows that (1) implies (4). Since a finite dimensional topologically irreducible representation on a normed space is a finite dimensional \ai representation, the existence of such a representation would imply the existence of a periodic point by Theorem~\ref{thm:description_finite_dimensional_representations}. Hence (4) implies (3), and trivially (3) implies (2). Since, by the second part of Theorem~\ref{thm:material_on_algebraically_irreducible_representations}, every \ai representation yields an algebraically (in particular: topologically) irreducible continuous representation on a Banach space, (2) implies (1).

The equivalence of (4) and (5) follows from \cite[Proposition~4.5]{Tnotes1}, taken together with Theorem~\ref{thm:cstar_reps}.
\end{proof}

Before proceeding with the main line, we give an application of what has been obtained so far. In \cite[Theorem~4.6]{DST}, it was established that $\lone$ is a Hermitian
algebra (i.e.\  that the spectrum of every selfadjoint element is a subset of the real numbers) whenever $\topspace$ is a finite set. We can now improve this.

\begin {theorem}\label{thm:Hermitian}
If all points of $\topspace$ are periodic, then $\lone$ is a Hermitian algebra.
\end {theorem}

\begin{proof}
As a consequence of \cite{P}, $\lone$ being Hermitian is equivalent to the following: If $\rep:\lone\to\linear(\vs)$ is an \ai representation of $\lone$ on a vector space $\vs$ and $e\in\vs$ is non-zero, then there exist a topologically irreducible \supast-representation $\rep^\prime:\lone\to \bounded(\hs)$ on a Hilbert space and a non-zero $h\in\hs$ such that $\{a\in\lone : \rep(a)e=0\}=\{a\in\lone :  \rep^\prime(a)h=0\}$. But this is obvious, since Proposition~\ref{prop:all_finite_dimensional} shows that every \ai representation is finite dimensional, and, by Theorem~\ref{thm:description_finite_dimensional_representations}, every such representation can be realised as a topologically irreducible \supast-representation on a Hilbert space.
\end{proof}

As an example, we know from Theorem~\ref{thm:Hermitian} that $\lone$ is Hermitian for rational rotations of $\T$. Later, we shall also determine the structure space of these algebras (see Example~\ref{ex:rotations}).

After this sidestep, we return to main line. In order to get further, we need to have more detailed information about the primitive ideals originating from Section~\ref{sec:representations}. While presenting this, we also introduce some terminology and notation that will be even more prominent in Section~\ref{sec:structure_space}.

The \ai representations in Section~\ref{sec:representations} are associated with points, but algebraically equivalent representations yield the same primitive ideal. We shall now take this into account, and describe the primitive ideals originating from Section~\ref{sec:representations} in a bijective fashion (see Proposition~\ref{prop:psi_proposition}). It turns out to be more natural to work with closures of orbits than with orbits, so, to establish notation, we let $\corbset$ denote the set of all closures of orbits in $\topspace$. Note that, as irrational rotations of $\T$ show, it may well be that different orbits have equal closure. The subset of (closures of) orbits of periodic points is denoted by $\corbsetperpoints$, and the subset of closures of orbits of aperiodic points by $\corbsetaperpoints$. Then $\corbset=\corbsetperpoints \sqcup  \corbsetaperpoints$ is a disjoint union.

Let $\pt\in\perpoints$ and let $\lambda\in\T$. If $y$ is in the orbit of $\pt$, then, by Proposition~\ref{prop:finite_dimensional_given}, the \ai representations $\rep_{\pt,\lambda}$ and $\rep_{y,\lambda}$ are algebraically equivalent. Hence we can associate a well defined primitive ideal $\prim_{\corb, \lambda}$ with $\lambda\in\T$ and the closure $\corb$ of the orbit $\orb$ of $\pt$, by putting $\prim_{\corb,\lambda}=\Ker(\rep_{\pt,\lambda})$. Of course, $\corb=\orb$ here, but, when we consider primitive ideals originating from aperiodic points, we shall see why the closure of an orbit is more natural than the orbit itself, and for consistency and ease of argumentation we employ this notation for finite orbits as well.

We have the following description of such $\prim_{\corb, \lambda}$.

\begin{proposition}\label{prop:finite_orbit_ideal}
Let $\corb\in\corbsetperpoints$ be \ulb the closure of\urb\ an orbit consisting of $p$-periodic points. If $a=\loneelement{n}\in\lone$, then for the primitive ideals associated with $\corb$ we have:
\begin{enumerate}
\item If $\lambda\in\T$, then $a\in\prim_{\corb, \lambda}$ if and only if
\begin{equation*}\label{e:vanishing}
\suml\lambda^l f_{lp + j}(\pt) = 0
\end{equation*}
for all $j=0,\ldots,p-1$ and all $\pt\in\orb$.
\item $a\in\bigcap_{\lambda\in\T}\prim_{\corb, \lambda}$ if and only if $f_n\rest{\corb}=0$ for all $n\in\Z$.
\end{enumerate}
\end{proposition}

For the proof we refer to \cite[Proposition~2.10]{JTBJMA}. It involves a straightforward computation for the first part, and then the second part follows easily from the injectivity of the Fourier transform on $\ell^1(\Z)$.

We now turn to the closure of an infinite orbit. The following  is immediately clear from the description of the corresponding representations $\rep_\pt^p$ from Section~\ref{subsec:infinite_dimensional} and the algebraic irreducibility of $\rep_\pt^1$, as asserted by Proposition~\ref{prop:one_algebraically_irreducible}. For the selfadjointness one need not resort to Theorem~\ref{thm:primitive_ideals_are_selfadjoint}, as this is clear by inspection.

\begin{lemma}\label{lem:infinite_orbit_ideal}
Let $\pt\in\aperpoints$, let $p\in[1,\infty]$, and let $a=\sumn f_n\delta^n\in\lone$. Then $\rep^p_\pt(a)=0$ if and only $f_n$\ vanishes on the closure of the orbit of $\pt$ for all $n
\in\Z$. Hence $\Ker (\rep_\pt^p)$ does not depend on $p$, and it is a selfadjoint primitive ideal of $\lone$.
\end{lemma}

The description of this common kernel in terms of the coefficients $f_n$ was already noted for $\pi_\pt^2$ in \cite[Proposition~2.10]{JTBJMA}. As a consequence of Lemma~\ref{lem:infinite_orbit_ideal}, if $\corb\in\corbsetaperpoints$, then we can take any $\pt\in \topspace$ such that $\overline{\Z\cdot \pt}=\corb$, and obtain a well defined primitive ideal by putting $\prim_{\corb}=\Ker(\rep_{\pt}^1)$. We include the above description in this notation for reference purposes.

\begin{proposition}\label{prop:infinite_orbit_ideal}
Let $\corb\in\corbsetaperpoints$ be the closure of an infinite orbit. Then for the primitive ideal associated with $\corb$ we have
\[
\prim_{\corb}=\left\{\loneelement{n}\in\lone : f_n\rest{\corb}=0 \textup{ for all } n \in \Z\right\}.
\]
\end{proposition}

Since $\topspace$ is the union of all orbit closures, the following result (already used in the proof of Theorem~\ref{thm:all_ideals_selfadjoint}) is now clear from Proposition~\ref{prop:finite_orbit_ideal} and Proposition~\ref{prop:infinite_orbit_ideal}.

\begin{theorem}\label{thm:semisimple}
\[
\bigcap_{\genfrac{}{}{0pt}{1}{\corb\in\corbsetperpoints}{\lambda\in\T}}\!\!\!\!\prim_{\corb,\lambda}\!\!\!\!\bigcap_{\genfrac{}{}{0pt}{1}{\corb\in\corbsetaperpoints}{}} \!\!\!\! \prim_{\corb}=\{0\}.
\]
In particular, the Banach algebra $\lone$ is semisimple.
\end{theorem}

We can now establish two separation results at the level of primitive ideals.

\newpage

\begin{proposition}\label{prop:finite_dimensional_separate}
The following are equivalent:
\begin{enumerate}
\item The \ai representations of $\lone$ on finite dimensional vector spaces separate the elements of $\lone$;
\item The \ai representations of $\lone$ on finite
dimensional vector spaces separate the elements of $\coeffalg$;
\item $Per(\homeo)$ is dense in $\topspace$;
\item The \ai representations of $C^{\ast}(\dynsysshort)$ on finite dimensional vector spaces separate the elements of $C^\ast(\dynsysshort)$.
\end{enumerate}
\end{proposition}

In view of the second part of Theorem~\ref{thm:material_on_algebraically_irreducible_representations}, one could also have considered the topologically irreducible continuous (or even contractive) representations on finite dimensional Banach spaces in (1) and (2).

\begin{proof}
From Theorem~\ref{thm:description_finite_dimensional_representations}, we know that the \ai finite dimensional representations are precisely the $\rep_{\pt,\lambda}$ from Section~\ref{subsec:finite_dimensional}, and, from Proposition~\ref{prop:finite_orbit_ideal}, we have $\bigcap_{\corb\in\corbsetperpoints, \lambda\in\T}P_{\corb,\lambda}=\{\sumn f_n\delta^n\in\lone: f_n\rest{\overline{\perpoints}}=0 \textup{ for all }n\in\Z\}$. It is now clear that (1), (2), and (3) are equivalent.

The equivalence of (3) and (4) follows from \cite[Theorem~4.6.(2)]{Tnotes1}, taken together with Theorem~\ref{thm:cstar_reps}.

\end{proof}

\begin{proposition}\label{prop:infinite_dimensional_separate}
The following are equivalent:
\begin{enumerate}
\item The \ai representations $\{\rep_\pt^1 : \pt\in\aperpoints\}$ of $\lone$ on infinite dimensional vector spaces separate the elements of $\lone$;
\item The \ai representations of $\lone$ on infinite
dimensional vector spaces separate the elements of $\lone$;
\item The \ai representations $\{\rep_\pt^1 : \pt\in\aperpoints\}$ of $\lone$ on infinite dimensional vector spaces separate the elements of $\coeffalg$;
\item The \ai representations of $\lone$ on infinite
dimensional vector spaces separate the elements of $\coeffalg$;
\item $\dynsysshort$ is topologically free;
\item The \ai representations of $C^{\ast}(\dynsysshort)$
on infinite dimensional vector spaces separate the elements of $C^\ast(\dynsysshort)$.
\end{enumerate}
\end{proposition}

\begin{proof}
From Proposition~\ref{prop:infinite_orbit_ideal}, we see that $\bigcap_{\corb\in\corbsetaperpoints}\prim_{\corb}=\{\sumn f_n\delta^n\in\lone : f_n\rest{\overline{\aperpoints}}=0\textup{ for all }n\in\Z\}$. Hence (5) implies (1). Trivially (1) implies (2), and (2) implies (4). Assume that (4) holds, and that $\{\rep_i\}_{i\in I}$ is a family of infinite dimensional \ai representations that separates the points of $\coeffalg$. Each of these has an induced dynamical system $\topspace_{\rep_i, \homeo_{\rep_i}}$ and an associated representation $\rep^\prime_i$ of $\ell^1(\Sigma_{\rep_i})$ such that $\rep_i(f)=\rep_i^\prime(f\rest{\topspace_{\rep_i}})$  for $f\in\coeffalg$. The separation property of $\{\rep_i\}_{i\in I}$ implies that $\bigcup_{i\in I}\topspace_{\rep_i}$ is dense in $\topspace$. Since the aperiodic points are dense in each $\topspace_{\rep_i}$ by Corollary~\ref{cor:properties_induced_system}, we can now conclude that $\aperpoints$ is dense in $\topspace$, i.e.\  that (5) holds. It is trivial that (1) implies (3) and that (3) implies (4).

The equivalence of (3) and (4) follows from \cite[Proposition~4]{T3}, taken together with Theorem~\ref{thm:cstar_reps}.
\end{proof}

Part (1), (3), and (4) of the following result on inclusions between primitive ideals are from \cite[Proposition~2.15]{JTBJMA}; we include the short proofs for the convenience of the reader.

\begin{lemma}\quad\label{lem:possible_inclusions}
\begin{enumerate}
 \item Let $\corb_1,\corb_2\in\corbsetperpoints$ and let $\lambda_1,\lambda_2\in\T$. If $\prim_{\corb_1,\lambda_1}\subset\prim_{\corb_2,\lambda_2}$, then $\corb_1=\corb_2$ and $\lambda_1=\lambda_2$. In particular, $\prim_{\corb,\lambda_1}=\prim_{\corb,\lambda_2}$.
 \item If $\corb\in\corbsetperpoints$, $\lambda\in\T$, $\pt\in\corb$, and $\rep$ is an \ai representation of $\lone$ such that $\prim_{\corb, \lambda}\subset \Ker(\rep)$, then $\rep$ is algebraically equivalent to $\rep_{\pt, \lambda}$. In particular,  $\prim_{\corb, \lambda}=\Ker(\rep)$.
 \item Let $\corb_1,\corb_2\in\corbsetaperpoints$. Then $\prim_{\corb_1}\subset\prim_{\corb_2}$ if and only if $\corb_1\supset\corb_2$.
 \item Let $\corb_1\in\corbsetaperpoints$, let $\corb_2\in\corbsetperpoints$, and let $\lambda\in\T$. Then $\prim_{\corb_1}\subset\prim_{\corb_2,\lambda}$ if and only if $\corb_1\supset\corb_2$.
 \end{enumerate}
 \end{lemma}

\begin{proof}
For (1), we recall that $P_{\corb, \lambda}=\Ker(\rep_{\pt,\lambda})$, where $x\in \topspace$ is such that $\corb=\overline{\Z\cdot x}$. If $\prim_{\corb_1,\lambda_1}\subset\prim_{\corb_2,\lambda_2}$, then in particular $\Ker(\rep_{\pt_1,\lambda_1}\rest{\coeffalg})\subset\Ker(\rep_{\pt_2,\lambda_2}\rest{\coeffalg})$. This implies that (and is equivalent to) $\corb_1\supset\corb_2$; hence $\corb_1=\corb_2$. Then also the representation spaces have the same dimension $p$. Since $\rep_{\pt_1,\lambda_1}(1-\delta^p/\lambda_1)=0$, we have  $\rep_{\pt_2,\lambda_2}(1-\delta^p/\lambda_1)=0$, and therefore $\lambda_1=\lambda_2$.

Turning to (2), we first note that $\lone/\Ker(\rep)$ is a quotient of $\lone/\prim_{\corb,\lambda}$. Hence we have an algebraically irreducible representation of $\lone/\prim_{\corb,\lambda}$ on the representation space of $\rep$. Since $\lone/\prim_{\corb,\lambda}$ is a finite dimensional algebra, $\rep$ is finite dimensional. Proposition~\ref{prop:structure_of_finite_dimensional_representations} then shows that $\rep$ is algebraically equivalent to $\rep_{y,\mu}$ for some $y\in\perpoints$ and $\mu\in\T$. An appeal to part (1) and Proposition~\ref{prop:finite_dimensional_given} then concludes the proof of part (2).

Part (3) is clear from Proposition~\ref{prop:infinite_orbit_ideal}.

Turning to part (4), if $\corb_1\supset\corb_2$, then Proposition~\ref{prop:finite_orbit_ideal} and Proposition~\ref{prop:infinite_orbit_ideal} show that $\prim_{\corb_1}\subset\prim_{\corb_2,\lambda}$. If $\corb_1\not\supset\corb_2$, then $\corb_1\cap\corb_2=\emptyset$, so that there exists $f\in\coeffalg$ such that $f\rest{\corb_1}=0$ and $f\rest{\corb_2}\neq 0$. Then $f\in\prim_{\corb_1}$, but $f\notin\prim_{\corb_2,\lambda}$; hence $\prim_{\corb_1}\not \subset\prim_{\corb_2,\lambda}$.
\end{proof}

\begin{remark}\label{rem:never_equal}
The argument used in the proof of the second part of Lemma~\ref{lem:possible_inclusions} shows that the kernel of a finite dimensional \ai representation of $\lone$ (or any other associative algebra) is never contained in (and, in particular, never equal to) the kernel of an infinite dimensional \ai representation. This elementary observation is needed in the proof of Theorem~\ref{thm:basic_topology}.
\end{remark}

We let $\sspace$ denote the set of all primitive ideals of $\lone$. The following is now clear from Lemma~\ref{lem:possible_inclusions}.

\begin{proposition}\label{prop:psi_proposition}\quad
\begin{enumerate}
 \item
Consider the natural map
 \[
 \Psi: \left(\corbsetperpoints\times\T\right) \sqcup  \corbsetaperpoints\to\sspace,
 \]
 defined by $\Psi((\corb,\lambda))=\prim_{\corb,\lambda}$  for $\corb\in\corbsetperpoints$ and $\lambda\in\T$, and by $\Psi(\corb)=\prim_{\corb}$ for $\corb\in\corbsetaperpoints$. Then $\Psi$ is injective.
 \item For all $\corb\in\corbsetperpoints$ and $\lambda\in\T$, the primitive ideal $\prim_{\corb,\lambda}$ \ is a maximal primitive ideal.
 \item For all $\corb\in\corbsetaperpoints$, the primitive ideal $\prim_{\corb}$ is a maximal element of the range of $\Psi$ if and only if $\corb$ is a minimal orbit closure.
\end{enumerate}
\end{proposition}

Thus $\Psi$ parameterizes the part of $\sspace$ that is naturally associated with the points of $\topspace$, and we shall denote this part by $\sspaceset{\topspace}$. It is always large in $\sspace$ in the sense that, as a consequence of Theorem~\ref{thm:semisimple}, it is dense in the hull-kernel topology of $\sspace$ to be studied in Section~\ref{sec:structure_space}. The subset of $\sspaceset{\topspace}$ that consists of the primitive ideals $\prim_{\corb,\lambda}$ for all $\corb\in\corbsetperpoints$ and $\lambda\in\T$ is denoted by $\sspaceset{\perpoints}$, and the subset that consists of the primitive ideals $\prim_{\corb}$ for all $\corb\in\corbsetaperpoints$ is denoted  by $\sspaceset{\aperpoints}$, so that $\sspaceset{\topspace}=\sspaceset{\perpoints}\sqcup\sspaceset{\aperpoints}$ is a disjoint union. We let $\sspaceinfdim$ denote the subset of $\sspace$ consisting of the kernels of all infinite dimensional \ai representations, so that $\sspace=\sspaceset{\perpoints}\sqcup\sspaceinfdim$ is a disjoint union.

It is remarkable that we can prove that $\sspaceset{\topspace}$ actually equals $\sspace$ if $\topspace$ is metrizable, even though we do not generally know all equivalence classes of \ai representations. This is a consequence of the following result, the proof of which once more illustrates the relevance of the technique of induced dynamical systems.

\begin{proposition}\label{prop:must_be_dynamical}
 Let $\rep$ be an \ai representation of $\lone$ on an infinite dimensional vector space. If the topological space $\inducedset$ is metrizable, then there exists an aperiodic point $\pt\in \inducedset$ such that its orbit is dense in $\inducedset$. Consequently, $\Ker(\rep) = \prim_{\corb}$, where $\corb$ is the closure of the orbit of $\pt$.
\end{proposition}

\begin{proof}
Let $\inducedsystemshort$ be the induced dynamical system. Since $\rep$ is infinite dimensional, Corollary~\ref{cor:properties_induced_system} shows that $\inducedsystemshort$ is topologically free, and hence Proposition~\ref{prop:self_adjoint} implies that $\Ker(\rep)$ is the closed ideal generated by $\{f\in\coeffalg : f\rest{\inducedset}=0\}$.

Since Corollary~\ref{cor:properties_induced_system} asserts that $\inducedset$ is topologically transitive, the metrizability of $\inducedset$ then implies, by \cite[Theorem~1.1.3]{T1}, that there exists a point $\pt\in \inducedset$ such that its orbit is dense in $\inducedset$. Since Corollary~\ref{cor:properties_induced_system} also asserts that $\inducedset$ is an infinite set, $\pt$ must be aperiodic. If we note that $\inducedset$ is closed, we see that $\inducedset$ is the closure in $\topspace$ of the orbit of $\pt$. The rest is clear.
\end{proof}

When observing that finite orbits are clearly minimal orbit closures, the following is now obvious from Propositions~\ref{prop:psi_proposition} and~\ref{prop:must_be_dynamical}; note that Proposition~ \ref{prop:all_finite_dimensional} asserts that there are no infinite dimensional \ai representations of $\lone$ at all if $\topspace=\perpoints$.

\begin{theorem}\label{thm:exhaustive} If, for every infinite dimensional \ai representation $\rep$ of $\lone$, the closed subset $\inducedset$ of $\topspace$ is metrizable \ulb this is certainly the case  if $\topspace$ is metrizable or $\topspace=\perpoints$\urb, then:
\begin{enumerate}
 \item $\sspaceinfdim=\sspaceset{\aperpoints}$, and hence $\sspace=\sspaceset{\topspace}$;
 \item A primitive ideal is a maximal primitive ideal if and only if it is associated with a minimal orbit closure.
 \end{enumerate}
\end{theorem}

\begin{remark}\label{rem:type_I}
To each algebraic equivalence class of \ai representations of $\lone$ one can assign the primitive ideal that is the common kernel of these representations. It follows from part (1) of Lemma~\ref{lem:possible_inclusions} and Proposition~\ref{prop:finite_dimensional_given} that this map is injective on the collection of algebraic equivalence classes of finite dimensional \ai representations of $\lone$.

On the collection of algebraic equivalence classes of infinite dimensional \ai representations, however, this map can be very far from injective. For example, for the irrational rotations of $\T$, where every orbit is infinite and dense, $\sspace$ consists only of the zero ideal. Even though Theorem~\ref{prop:equivalences} shows that different orbits provide different equivalence classes of \ai representations, the associated primitive ideals are always equal to the zero ideal. Proposition~\ref{prop:degenerate} below describes when this happens for general systems.
\end{remark}

\section{The structure space}\label{sec:structure_space}

We shall now consider the structure space of $\lone$, i.e.\ $\sspace$ in its hull-kernel topology. The main goal is Theorem~\ref{thm:main_homeomorphism}, asserting that, under suitable conditions, parts of $\sspace$ are homeomorphic to products of spaces of finite orbits and $\T$.

We recall from \cite[\S 26]{BD} that, for $E\subset\sspace$, the closure of $E$ in this topology is the hull $\hull\kernel(E)$ of the kernel $\kernel(E)$, where $\kernel(E)=\bigcap_{\prim\in E}P$, and $\hull\kernel(E)=\{P\in\sspace : P\supset\kernel(E)\}$. Clearly, $E$ is dense in $\sspace$ if and only if $\kernel(E)$ equals the Jacobson radical $\bigcap_{\prim\in\sspace} P$ of $\lone$. Since this is the zero ideal by Theorem~\ref{thm:semisimple}, $E$ is dense if and only if $\kernel(E)=\{0\}$. Thus, for example, if there exists an aperiodic point with dense orbit, then the associated singleton $\{\prim_{\corb}\}=\{\{0\}\}$ is a dense subset of $\sspace$.

We shall now rephrase some of the results in Section~\ref{sec:primitive_ideals} in terms of $\sspace$ and its topology. Here, and elsewhere, subsets of $\sspace$ are supplied with the induced topologies from $\sspace$ unless otherwise stated. Proposition~\ref{prop:topological_complements} contains two more results in this vein.

\begin{theorem}\label{thm:basic_topology}\quad
 \begin{enumerate}
  \item $\sspace$ is compact.
  \item $\sspaceset{\topspace}$ is dense in $\sspace$.
  \item The following are equivalent:
  \begin{enumerate}
  \item $\sspaceset{\perpoints}=\sspace$;
  \item $\perpoints=\topspace$.
 \end{enumerate}
  \item The following are equivalent:
  \begin{enumerate}
  \item $\sspaceset{\perpoints}$ is dense in $\sspace$;
  \item $\perpoints$ is dense in $\topspace$.
  \end{enumerate}
 \item The following are equivalent:
  \begin{enumerate}
  \item $\sspaceinfdim=\sspace$;
  \item $\aperpoints=\topspace$.
 \end{enumerate}
  \item The following are equivalent:
  \begin{enumerate}
  \item $\sspaceset{\aperpoints}$ is dense in $\sspace$;
  \item $\sspaceinfdim$ is dense in $\sspace$;
  \item $\aperpoints$ is dense in $\topspace$.
  \end{enumerate}
  \item
  \begin{enumerate}
    \item For all $\corb\in\corbsetperpoints$ and $\lambda\in\T$, the singleton $\{\prim_{\corb,\lambda}\}$ is closed in $\sspace$.
    \item For all $\corb\in\corbsetaperpoints$,  the singleton $\{\prim_{\corb}\}$ is closed in $\sspaceset{\topspace}$ if and only if $\corb$ is a minimal orbit closure.
    \end{enumerate}
  \item  If, for every infinite dimensional \ai representation $\rep$ of $\lone$, the closed subset $\inducedset$ of $\topspace$ is metrizable \ulb this is certainly the case if $\topspace$ is metrizable or $\topspace=\perpoints$\urb, then:
  \begin{enumerate}
  \item $\sspaceset{\topspace}=\sspace$, and $\sspaceset{\topspace}$ is compact;
  \item If $\prim\in\sspace$, then the singleton $\{\prim\}$ is closed in $\sspace$ if and only if $\prim$ is associated with a minimal orbit closure.
  \end{enumerate}

 \end{enumerate}
\end{theorem}

\begin{proof}
Part (1) follows from the fact that $\lone$ is a unital algebra (see \cite[Corollary~26.5]{BD}). Part (2) through (7) follow from Theorem~\ref{thm:semisimple}, Proposition~\ref{prop:all_finite_dimensional},
 Proposition~\ref{prop:finite_dimensional_separate}, Proposition~\ref{prop:all_infinite_dimensional}, Proposition~\ref{prop:infinite_dimensional_separate}, Proposition~\ref{prop:psi_proposition}, and Theorem~\ref{thm:exhaustive}, respectively, taking Remark~\ref{rem:never_equal} and Theorem~\ref{thm:description_finite_dimensional_representations} into account where necessary.
Part (8)(a) follows from Theorem~\ref{thm:exhaustive}.(1) and part (1). Noting that finite orbits are minimal orbit closures, part (8)(b) follows from the fact that $\sspaceset{\topspace}=\sspace$, combined with part (7).
\end{proof}

Before we proceed, we note that we can determine when $\sspaceset{\topspace}$ degenerates to its minimal size of one point: this occurs precisely for infinite minimal systems. Note that, for metrizable $\topspace$, this is equivalent to the degeneracy of the whole primitive ideal space $\sspace= \sspaceset{\topspace}$.

\begin{proposition}\label{prop:degenerate}
The following are equivalent:
\begin{enumerate}
 \item $\sspaceset{\topspace}$ consists of one point;
 \item $\sspaceset{\topspace}=\{\{0\}\}$;
 \item $\topspace=\aperpoints$, and every orbit is dense.
\end{enumerate}

\end{proposition}

\begin{proof}
 It follows from Theorem~\ref{thm:semisimple} that (1) implies (2). If (2) holds, then all points must be aperiodic, since a periodic point yields non-zero primitive ideals; it is then clear that every orbit must be dense. Hence (2) implies (3), and it is obvious that (3) implies (1).
\end{proof}

Continuing with the main line, we consider natural subsets of $\sspace$ that are associated with more general invariant subsets of $\topspace$ than $\topspace$, $\perpoints$, or $\aperpoints$, as we have done so far. With notation consistent with that already introduced, we define, for  invariant $S\subset\topspace$,
\[
\sspaceset{S}=\{\prim_{\corb,\lambda} : \corb\in\corbsetperpoints,\, \lambda\in\T,\, \orb\subset S\}\,\cup\,\{P_{\corb} : \corb\in\corbsetaperpoints,\, \orb\subset S\}.
\]
Hence $\sspaceset{S}$ consists of all primitive ideals associated with the closures of all orbits contained in $S$; note that these closures themselves need not be contained in $S$.

For further investigation of such subsets of $\sspace$, the following lemma is convenient.

\begin{lemma}\label{lem:closure}
Let $S\subset\topspace$ be invariant. Then:
\begin{enumerate}
 \item If $\corb\in\corbsetperpoints$ and $\lambda\in\T$, then $\{\loneelement{n}\in\lone : f_n\rest{S}=0 \textup{ for all } n\in\Z\}\subset\prim_{\corb,\lambda}$ if and only if $\corb\subset \overline{S}$.
 \item If $\corb\in\corbsetaperpoints$, then $\{\loneelement{n}\in\lone : f_n\rest{S}=0 \textup{ for all } n\in\Z\}\subset\prim_{\corb}$ if and only if $\corb\subset \overline{S}$.
\end{enumerate}
\end{lemma}

\begin{proof}
For (1), if $\corb\not\subset \overline{S}$, then the invariance of $S$ and the finiteness of $\corb $ imply that $\corb\cap \overline S=\emptyset$. Hence there exists $f\in\coeffalg$ such that $f\rest{\overline S}=0$ and $f\rest{\corb}\neq 0$. Then $f\in\{\loneelement{n}\in\lone : f_n\rest{S}=0 \textup{ for all } n\in\Z\}$, but $f\notin\prim_{\corb,\lambda}$. Hence $\{\loneelement{n}\in\lone : f_n\rest{S}=0 \textup{ for all } n\in\Z\}\not\subset\prim_{\corb}$. The converse implication in (1) is clear. Part (2) is obvious.
\end{proof}

For invariant $S\subset \topspace$, we can now describe the closure of $\sspaceset{S}$ in $\sspaceset{\topspace}$.

\begin{proposition}\label{prop:closure}\quad
\begin{enumerate}
\item Let $S\subset\topspace$ be invariant. Then the closure of $\sspaceset{S}$ in $\sspaceset{\topspace}$ is $\sspaceset{\overline S}$.
\item Let $S_1\subset S_2\subset\perpoints$ be two invariant subsets. Then:
\begin{enumerate}
\item The closure of $\sspaceset{S_1}$ in $\sspaceset{S_2}$ is $\sspaceset{{\overline S_1}^{S_2}}$, where ${\overline S_1}^{S_2}$ is the closure of $S_1$ in $S_2$.
\item $\sspaceset{S_1}$ is closed in $\sspaceset{S_2}$ if and only if $S_1$ is closed in $S_2$.
\end{enumerate}
\end{enumerate}
\end{proposition}

\begin{proof}
For the first statement, we note that it follows from the second part of Proposition~\ref{prop:finite_orbit_ideal}, Proposition~\ref{prop:infinite_orbit_ideal}, and continuity, that
\begin{align*}
\kernel(\sspaceset{S})&=\kernel\left(\left\{\prim_{\corb,\lambda} : \corb\in\corbsetperpoints, \lambda\in\T,\orb\subset S\right\}\cup\left\{P_{\corb} : \corb\in\corbsetaperpoints,\orb\subset S\right\}\right)\\
&=\left\{\loneelement{n}\in\lone : f_n\rest{\overline{\bigcup_{\orb\subset S}\corb}}=0 \textup{ for all } n\in\Z\right\}.
\end{align*}
But $\overline{\bigcup_{\orb\subset S}\corb}=\overline{S}$, and hence Lemma~\ref{lem:closure} shows that
\begin{align*}
\hull\kernel(\sspaceset{S})  \cap\sspaceset{\topspace} & = \left\{\prim_{\corb,\lambda} : \corb\in\corbsetperpoints ,\, \lambda \in\T,\corb\subset\overline{S}\right\}\\
&\quad\quad\cup\left\{\prim_{\corb}  : \corb\in\corbsetaperpoints,\,\corb\subset\overline{S}\right\}.
\end{align*}
Since $\overline{S}$ is closed, an orbit is contained in $\overline{S}$ precisely when its closure is contained in $\overline{S}$; hence the right hand side equals $\sspaceset{\overline{S}}$, as claimed.
\\ For the second part, we note that, if $S_1\subset S_2\subset\perpoints$ are two invariant subsets, then
\begin{align*}
\hull\kernel(\sspaceset{S_1})\cap\sspaceset{S_2}&=\hull\kernel(\sspaceset{S_1})\cap\sspaceset{\topspace}\cap\sspaceset{S_2}
\\&=\left\{\prim_{\corb,\lambda} : \corb\in\corbsetperpoints,\, \lambda\in\T,\,\corb\subset\overline{S_1}\right\}\cap\sspaceset{S_2}
\\&=\left\{\prim_{\corb,\lambda} : \corb\in\corbsetperpoints,\, \lambda\in\T,\,\orb\subset\overline{S_1}\right\}\\
& \quad\quad\cap \left\{\prim_{\corb,\lambda} : \corb\in\corbsetperpoints,\, \lambda\in\T,\,\orb\subset S_2\right\}
\\&=\left\{\prim_{\corb,\lambda} : \corb\in\corbsetperpoints,\, \lambda\in\T,\,\orb\subset\overline{S_1}\cap S_2\right\}\\&= \sspaceset{{\overline S_1}^{S_2}}.
\end{align*}
This proves the first statement in the second part. For the second statement we need then merely note that the map $S\mapsto\sspaceset{S}$ is injective on the collection of invariant subsets of $\perpoints$, as a direct consequence of the first part of Lemma~\ref{lem:possible_inclusions}. (Note, for the sake of completeness, that\textemdash as irrational rotations of $\T$ show\textemdash this is not generally true for the collection of invariant subsets of $\aperpoints$.)
\end{proof}

\begin{corollary}\label{cor:perpoints_closed}
Let $S\subset\perpoints$ be invariant.
Then the following are equivalent:
\begin{enumerate}
\item $\sspaceset{S}$ is closed in $\sspaceset{\topspace}$;
\item $S$ is closed in $\topspace$.
\end{enumerate}
\end{corollary}

\begin{proof}
Assume that (1) holds. If $S$ is not closed in $\topspace$, then there is an orbit $\orb$ in $\overline{S}$ such that $\orb\not\subset S$. If this is an infinite orbit, then trivially $\prim_{\corb}\notin\sspaceset{S}$. The first part of Proposition~\ref{prop:closure}, however, implies that $\prim_{\corb}$ is in the closure of $\sspaceset{S}$ in $\sspaceset{\topspace}$. This contradicts the fact that $\sspaceset{S}$ is closed in $\sspaceset{\topspace}$. If $\orb$ is a finite orbit, then $\corb=\orb\not\subset S$, and hence, by the first part of Lemma~\ref{lem:possible_inclusions}, $\prim_{\corb,\lambda}\notin\sspaceset{S}$  for all $\lambda\in\T$. The first part of Proposition~\ref{prop:closure} implies again that $\prim_{\corb,\lambda}$ is in the closure of $\sspaceset{S}$ in $\sspaceset{\topspace}$, which is again a contradiction. Hence $S$ must be closed in $S$, and (1) implies (2). It is immediate from the first part of Proposition~\ref{prop:closure} that (2) implies (1).
\end{proof}

We can now establish the following two additions to the results listed in Theorem~\ref{thm:basic_topology}.

\begin{proposition}\label{prop:topological_complements}\quad
\begin{enumerate}
 \item The following are equivalent:
 \begin{enumerate}
 \item $\sspaceset{\perpoints}$ is closed in $\sspaceset{\topspace}$;
 \item $\perpoints$ is closed in $\topspace$.
 \end{enumerate}
 \item The following are equivalent:
 \begin{enumerate}
 \item $\sspaceset{\aperpoints}$ is closed in $\sspaceset{\topspace}$;
 \item $\aperpoints$ is closed in $\topspace$.
 \end{enumerate}
\end{enumerate}
\end{proposition}

\begin{proof}
 Part (1) is a special case of Corollary~\ref{cor:perpoints_closed}.
 For part (2), it is clear from the first part of Proposition~\ref{prop:closure} that (b) implies (a). If (a) holds, but $\aperpoints$ is not closed in $\topspace$, choose a finite orbit $\corb$ in  $\perpoints\cap\overline{\aperpoints}$. Trivially, for $\lambda\in\T$, $\prim_{\corb,\lambda}\notin\sspaceset{\aperpoints}$, but the first part of Proposition~\ref{prop:closure} shows that, for $\lambda\in\T$, $\prim_{\corb,
 \lambda}$ is in the closure of $\sspaceset{\aperpoints}$ in $\sspaceset{\topspace}$  . This contradicts the fact that $\sspaceset{\aperpoints}$ is closed in $\sspaceset{\topspace}$.
\end{proof}

If $S\subset \perpoints$ is invariant, we let $\orbspace{S}$ be the associated orbit space, supplied with the quotient topology. For the remainder of this section, we shall concentrate, for suitable invariant subsets $S$ of $\perpoints$, on describing the topology of $\sspaceset{S}$ in terms of the topological product $\orbspace{S}\times\T$ (see Theorem~\ref{thm:main_homeomorphism}). If $\orb=\corb\subset S$ is an orbit, we shall write $\corb$ for the subset of $S$ as well as for the corresponding element of $\orbspace{S}$.

The following result on a restriction of the inverse of $\Psi$ in Proposition~\ref{prop:psi_proposition} (that can be defined on the whole of $\sspaceset{\topspace}$) relies on a non-trivial result in Fourier analysis.

\begin{lemma}\label{lem:continuity} Let $p\geq 1$ be an integer, and suppose that $S\subset\pperpoints$ is invariant.  Then the restricted inverse map $\Psi^{-1}:\sspaceset{S}\to\orbspace{S}\times\T$, sending $\prim_{\corb,\lambda}\in\sspaceset{S}$ to $(\corb,\lambda)\in\orbspace{S}\times\T$, is a continuous bijection.
\end{lemma}

\begin{proof}
It is clear that the map is a bijection, and it remains to show that it is continuous.
Let $q_1,q_2$ be the canonical projections from $\orbspace{S}\times\T$ onto the first and second factor, respectively. We are to show that $q_1\circ\Psi^{-1}$ and $q_2\circ\Psi^{-1}$ are continuous.

If $F\subset\orbspace{S}$ is closed, then there exists an invariant $S_F\subset S$ that is closed in $S$, and such that $F=\{\corb: \corb\subset S_F\}$.  Hence
\[
(q_1\circ\Psi^{-1})^{-1}(F)=\left\{\prim_{\corb,\lambda} : \lambda\in\T,\,\corb\subset S\right\}
=\sspaceset{S_F}.
\]
Since $S_F$ is closed in $S$, part (2)(b) of Proposition~\ref{prop:closure} implies that $\sspaceset{S_F}$ is closed in $\sspaceset{S}$. Hence $q_1\circ\Psi^{-1}$ is continuous.

If $F^\prime\subset\T$, then $\left(q_2\circ\Psi^{-1}\right)^{-1}(F^\prime)=\{\prim_{\corb,\lambda} :  \lambda\in F^\prime,\, \corb\subset S\}$. If $F^\prime=\emptyset$ or $F^\prime=\T$, then this equals $\emptyset$ or $\sspaceset{S}$, respectively;  hence it is closed in $\sspaceset{S}$. If $F^\prime\subset\T$ is a non-trivial closed subset, we argue as follows. Proposition~\ref{prop:finite_orbit_ideal} shows that
\begin{align*}
&\kernel\left(\{\prim_{\corb,\lambda} : \lambda\in F^\prime,\,\corb\subset S\}\right)=\\
&=\left\{\loneelement{n}\in\lone:   \suml\lambda^l f_{lp + j}(\pt) = 0\!\textup{ for }j=0,\ldots,p-1, \pt\in S,\!\textup{ and }\!\lambda\in F^\prime\right\}.
\end{align*}
Suppose that $\lambda_0\notin F^\prime$. Then, as a consequence of \cite[Theorem~7.1.2.(iii)]{L}, the regularity of the Banach algebra $\ell^1(\Z)$, and the well known fact that the maximal ideal space of $\ell^1(\Z)$ is $\T$ with the usual topology, there exists $(c_n)\in\ell^1(\Z)$ such that $\sumn\lambda^n c_n=0$ for all $\lambda\in F^\prime$ and $\sumn\lambda_0^n c_n=1$. Let $a_{lp+j}=c_l$ for $j=0,\ldots,p-1$ and $l\in\Z$, and put $a=\sumn a_n \delta^n\in\lone$. Then $a\in \kernel\left(\{\prim_{\corb,\lambda} : \lambda\in F^\prime,\, \corb\subset S \}\right)$, but $a\notin P_{\corb,\lambda_0}$ for all $\corb\subset\pperpoints$, and in particular for all $\corb\subset S$. Hence $\kernel\left(\{\prim_{\corb,\lambda} :  \lambda\in F^\prime,\,\corb\subset S\}\right)\not\subset P_{\corb,\lambda_0}$ for all $\corb\subset S$, i.e.\  $\prim_{\corb,\lambda_0}\notin\hull\kernel\left(\{\prim_{\corb,\lambda} : \lambda\in F^\prime,\,\corb\subset S\}\right)$ for all $\corb\subset S$. Since $\lambda_0\notin F^\prime$ was arbitrary, we conclude that $\hull\kernel(\left(q_2\circ\Psi^{-1}\right)^{-1}(F^\prime))\cap\sspaceset{S}=\left(q_2\circ\Psi^{-1}\right)^{-1}(F^\prime)$, i.e.\  that $\left(q_2\circ\Psi^{-1}\right)^{-1}(F^\prime)$\ is closed in $\sspaceset{S}$. Hence $q_2\circ\Psi^{-1}$ is continuous.
\end{proof}

\begin{corollary}\label{cor:basic_homeomorphism} Let $p\geq 1$ be an integer. Suppose that $S\subset\pperpoints$ is invariant and that $\sspaceset{S}$ is compact. Then the restricted map $\Psi^{-1}:\sspaceset{S}\to\orbspace{S}\times\T$, sending $\prim_{\corb,\lambda}\in\sspaceset{S}$ to $(\corb,\lambda)\in\orbspace{S}\times\T$, is a homeomorphism between compact Hausdorff spaces.
\end{corollary}

\begin{proof}
The space $\orbspace{S}$ is Hausdorff. Indeed, if $\orb_1,\orb_2$ are two different orbits in $S$, then there exist disjoint (relatively) open subsets $U_1$ and $U_2$ of $S$ such that $U_1\supset \orb_1$ and $U_2\supset\orb_2$. Then $\bigcap_{n\in\Z}\homeo^n(U_1)=\bigcap_{n=0}^{p-1}\homeo^n(U_1)$ and $\bigcap_{n\in\Z}\homeo^n(U_2)=\bigcap_{n=0}^{p-1}\homeo^n(U_2)$ are two disjoint invariant open subsets of $S$ separating $\orb_1$ and $\orb_2$. The images under the (open) quotient map then provide two open disjoint subsets of $\orbspace{S}$ separating $\orb_1$ and $\orb_2$, as required.

By Lemma~\ref{lem:continuity}, we see that $\Psi^{-1}:\sspaceset{S}\to\orbspace{S}\times\T$ is a continuous bijection between a compact space and a Hausdorff space. Hence it is a homeomorphism.
\end{proof}

We can now describe the topology of $\sspaceset{S}$ for certain invariant $S\subset\perpoints$.

\begin{theorem}\label{thm:technical_main_theorem}
Suppose that $S_1,\ldots,S_n\subset\perpoints$ are mutually disjoint invariant subsets, such that  $S_i\subset\per_{p_i}(\homeo)$ for not necessarily different integers $p_1,\ldots,p_n\geq 1$. Let $S=\bigcup_{i=1}^n S_i$, so that $\sspaceset{S}=\bigcup_{i=1}^n \sspaceset{S_i}$ as a disjoint union of sets. If each $S_i$ is closed in $S$, and $\sspaceset{S}$ is compact, then the map $\Psi: \bigsqcup_{i=1}^n \left(\orbspace{S_i}\times \T\right)\to\sspaceset{S}$, sending $(\corb,\lambda)$ to $P_{\corb,\lambda}$, is a homeomorphism between compact Hausdorff spaces. Here  $\bigsqcup_{i=1}^n \left(\orbspace{S_i}\times \T\right)$ is the disjoint union of the topological spaces $\orbspace{S_i}\times \T$.
\end{theorem}

\begin{proof}
Since each $S_i$ is closed in $S$, part (2)(b) of Proposition~\ref{prop:closure} shows that each $\sspaceset{S_i}$ is closed in $\sspaceset{S}$. Since the latter space is compact by assumption, each $\sspaceset{S_i}$ is compact. Hence Corollary~\ref{cor:basic_homeomorphism} applies, and it  shows that each $\sspaceset{S_i}$ is a compact Hausdorff space. We are now in the situation of part (1) of Lemma~\ref{lem:topology}, where $\sspaceset{S}=\bigcup_{i=1}^n \sspaceset{S_i}$ is the finite disjoint union of the subsets $\sspaceset{S_i}$ of $\sspaceset{S}$, and each $\sspaceset{S_i}$ is a closed subset of $\sspaceset{S}$ that is a compact Hausdorff space in the induced topology. Hence $\sspaceset{S}$ is a compact Hausdorff space, and it is the disjoint union of the topological spaces $\sspaceset{S_i}$. Since Corollary~\ref{cor:basic_homeomorphism} shows that each $\sspaceset{S_i}$ is homeomorphic to $\orbspace{S_i}\times \T$, the proof is complete.
\end{proof}

An application of the second and then the third part of Lemma~\ref{lem:topology} shows that $\sspaceset{S}$ is homeomorphic to $\left(\orbspace{\bigsqcup_{i=1}^n S_i\right)}\times\T$, where $\bigsqcup_{i=1}^n S_i$ is the topological disjoint union of the $S_i$. This will be used in the proof of the following main result on the topological structure of a part of $\sspace$.

\begin{theorem}\label{thm:main_homeomorphism}
Suppose that $\sspaceset{\topspace}$ is compact; this is certainly the case if $\topspace$ is metrizable or $\topspace=\perpoints$.  Furthermore, assume that $S_1,\ldots,S_n$ are mutually disjoint invariant closed subsets of $\topspace$ such that $S_i\subset\per_{p_i}(\homeo)$ for not necessarily different integers $p_1,\ldots,p_n\geq 1$. Let $S=\bigcup_{i=1}^n S_i$.
 Then the map $\Psi: \orbspace{S}\times \T\to\sspaceset{S}$, sending $(\corb,\lambda)$ to $P_{\corb,\lambda}$, is a homeomorphism between compact Hausdorff spaces. As a topological space, the codomain is also homeomorphic to the disjoint union of the topological spaces $\sspaceset{S_i}$, and each such space is homeomorphic to $\orbspace{S_i}\times\T$.
\end{theorem}

\begin{proof}
The sufficiency of the conditions in the first sentence for $\sspaceset{\topspace}$ to be compact follows from part (8) of Theorem~\ref{thm:basic_topology}. We turn to the remaining statements. Since the $S_i$ are now closed, $S$ is closed, and then Corollary~\ref{cor:perpoints_closed} shows that $\sspaceset{S}$ is closed in $\sspaceset{\topspace}$. The latter space is compact by assumption, so that $\sspaceset{S}$ is compact. Therefore Theorem~\ref{thm:technical_main_theorem} applies. Combining this with the remark following that theorem, we see that $\sspaceset{S}$ is homeomorphic to $\left(\orbspace{\bigsqcup_{i=1}^n S_i\right)}\times\T$, where $\bigsqcup_{i=1}^n S_i$ is the topological disjoint union of the $S_i$. An application of the first part of Lemma~\ref{lem:topology} shows that $\bigsqcup_{i=1}^n S_i$ is homeomorphic to $S$. This completes the proof.
\end{proof}

We conclude with two special cases in which there are homeomorphisms as in Theorem~\ref{thm:main_homeomorphism}. With Remark~\ref{rem:type_I} in mind, the first one can be regarded as an improved version (with topology added) of part of Theorem~\ref{thm:finite_set}.

\begin{corollary}
Suppose that $\topspace$ is a finite set. Then the structure space $\sspace$ of $\lone$ is homeomorphic to the topological disjoint union of copies of $\T$, one for each orbit in $\topspace$.
\end{corollary}

\begin{corollary}\label{cor:all_same_order} Suppose that all orbits in $\topspace$ are of the same finite order. Then the structure space $\sspace$ of $\lone$ is homeomorphic to $\orbspace{\topspace}\times\T$.
\end{corollary}

\begin{example}[Rotations of $\T$]\label{ex:rotations}
Let $\topspace=\T$ and let $\homeo$ be the rotation by $2\rep p/q$, where $p,q$ are integers such that $q\neq 0$ and having greatest common divisor equal to 1. Corollary~\ref{cor:all_same_order} shows that $\sspace$ is homeomorphic to $\orbspace{\T}\times \T$.  For each $z_0\in\T$, the orbit of $z_0$ consists of all $z\in\T$ such that $z^q=z_0^q$.  The latter implies that the map $z\mapsto z^q$ from $\T$ onto $\T$ induces a homeomorphism between $\orbspace{\T}$ and $\T$. We conclude that $\sspace$ is homeomorphic to $\T^2$.

For irrational rotations we had already seen in Remark~\ref{rem:type_I} that $\sspace=\{\{0\}\}.$
\end{example}

\begin{remark}\label{rem:kodaka}
We are not aware of results for the structure space of $\cstar$ that are the analogues of those for $\lone$ in the present section. The algebra $\lone$ is very concretely given, and this makes it more accessible to explicit computations than $\cstar$. For $\cstar$, one could conceivably use the generalized Fourier coefficients of its elements as substitutes for the coefficients of the elements of $\lone$ to work with. 

As first evidence that such an approach might be successful, we mention that, for rational rotations of $\T$, the structure space of $\cstar$ is known to be homeomorphic to $\T^2$, as would also follow from the $\cstar$-analogue of Corollary~\ref{cor:all_same_order}. Indeed, $\cstar$ is strongly Morita equivalent to $C(\T^2)$ (see \cite{Rieffel}), and therefore its structure space is homeomorphic to that of the latter algebra by \cite[Corollary~3.30]{RW}, i.e.\ to~$\T^2$.

Furthermore, the product of a space of orbits and $\T$ occurs in the following context. Let $p\geq 1$ be an integer, and let $\mathrm{Irr}_p(\dynsysshort)$ be the set of all irreducible unitary representations of $\cstar$ on a fixed Hilbert space of dimension $p$, supplied with the topology of pointwise strong convergence. Let $\widehat{A}_p(\dynsysshort)$ be the set of unitary equivalence classes of irreducible unitary representations of $\cstar$ of dimension $p$, supplied with the quotient topology originating from $\mathrm{Irr}_p(\dynsysshort)$. Then, as in Theorem~\ref{thm:description_finite_dimensional_representations}, there is natural bijection $\Xi_p$ between 
$\pperpoints/\Z\times\T$ and $\widehat{A}_p(\dynsysshort)$. According to \cite[Theorem~4.2.1]{T1}, this map is a homeomorphism. This result is in the same spirit as Theorem~\ref{thm:main_homeomorphism}, but it does not involve the hull-kernel topology of a part of the primitive ideal space as such. 

We leave the hull-kernel topology on the primitive ideal space of $\cstar$ for further research.
\end{remark}

\subsection*{Acknowledgements}
The authors are indebted to Kazunori Kodaka for pointing out the result on rotation algebras in Remark~\ref{rem:kodaka}, and to the referee for the very precise reading of the manuscript.

This work was partially supported by the Netherlands Organisation for Scientific Research (NWO).


\begin{thebibliography}{99}


\bibitem{AK} F.\ Albiac and N.J.\ Kalton, \emph{Topics in Banach space theory}, Springer, New York, 2006.

\bibitem {BD} F.F.\ Bonsall and J.\ Duncan, \emph{Complete normed algebras}, Springer Verlag, New York-Heidelberg, 1973.

\bibitem{DK2008} K.R.\ Davidson and E.G.\ Katsoulis, \emph{Isomorphisms between topological conjugacy algebras}, J.\ Reine Angew.\ Math.\ {\bf 621} (2008), 29--51.

\bibitem{DK2009} K.R.\ Davidson and E.G.\ Katsoulis, \emph{Nonself-adjoint operator algebras for dynamical systems}, in \emph{Operator Structures and Dynamical Systems}, edited by M.\ de Jeu, S.\ Silvestrov, C.\ Skau, and J.\ Tomiyama, Contemporary Mathematics {\bf 503} (2009), 39-51.

\bibitem{DK2011} K.R.\ Davidson and E.G.\ Katsoulis, \emph{Operator algebras for multivariable dynamics}, Mem.\ Amer.\ Math.\ Soc.\ {\bf 209} (2011), no.\ 982.

\bibitem{DdJW} S.\ Dirksen, M.\ de Jeu, M.\ Wortel, \emph{Crossed products
of Banach algebras.\ I}, 2011, to appear in Dissertationes Math., arXiv:1104.5151.

\bibitem {D} J.\ Dixmier, \emph{\Cstar-algebras}, North-Holland Publishing Co., Amsterdam-New York-Oxford, 1977.

\bibitem{dJMW} M.\ de Jeu, M.\ Messerschmidt and M.\ Wortel, \emph{Crossed
products of Banach algebras.\ II}, 2013, to appear in Dissertationes Math., arXiv:1305.2304.

\bibitem{dJM} M.\ de Jeu and M.\ Messerschmidt, \emph{Crossed products of
Banach algebras.\ III}, 2013, to appear in Dissertationes Math., arXiv:1306.6290.

\bibitem {DST} M.\ de Jeu, C.\ Svensson and J.\ Tomiyama, \emph{On the Banach \supast-algebra crossed product associated with a topological dynamical system}, J.\ Funct.\ Anal.\ {\bf 262} (2012), 4746--4765.

\bibitem {JTStudia} M.\ de Jeu and J.\ Tomiyama, \emph{Maximal abelian subalgebras and projections in two Banach algebras associated with a topological dynamical system}, Studia Math. {\bf 208} (2012), 47--75.

\bibitem{JTBJMA} M.\ de Jeu and J.\ Tomiyama, \emph{Noncommutative spectral synthesis for the involutive Banach algebra associated with a topological dynamical system}, Banach J.\ Math.\ Anal.\ {\bf 7} (2013), 103--135.

\bibitem {K} R.V.\ Kadison, \emph{Irreducible operator algebras}, Proc.\ Nat.\ Acad.\ Sci.\ U.S.A.\ {\bf 43} (1957), 273--276.

\bibitem{KishiTomi} A.\ Kishimoto and J.\ Tomiyama, \emph{Topologically irreducible representations of the Banach $^\ast$-algebra associated with a dynamical system}, 2016, arXiv:1604.02539v1.

\bibitem{L} R.~Larsen, \emph{Banach algebras. An introduction}, Marcel Dekker, New York, 1973.

\bibitem {P} T.W.\ Palmer, \emph{Hermitian Banach \supast-algebras}, Bull.\  Amer.\ Math.\ Soc.\ {\bf 78} (1972), 522--524.

\bibitem{Pbook} T.W.\ Palmer, \emph{Banach algebras and the general theory of \supast-algebras. Vol. I. Algebras and Banach algebras}, Cambridge University Press, Cambridge, 1994.

\bibitem{Phillips2012} N.C.\ Phillips, \emph{Analogs of Cuntz algebras on $L^p$-spaces}, 2012, arXiv:1201.4196.

\bibitem{Phillips2013a} N.C.\ Phillips, \emph{Simplicity of UHF and Cuntz algebras on $L^p$-spaces}, 2013, arXiv:1309.0115.

\bibitem{Phillips2013b} N.C.\ Phillips, \emph{Crossed products of $L^p$ operator algebras and the K-theory of Cuntz algebras on $L^p$ spaces},  2013, arXiv:1309.6406.

\bibitem {PT} V.\ Pt\'ak, \emph{Banach algebras with involution}, Manuscripta Math.\ {\bf 6} (1972), 245--290.

\bibitem{RW} I.\ Raeburn and D.P.\ Williams, \emph{Morita equivalence and continuous-trace $C^\ast$-algebras}, Math.\ Surveys and Monographs, vol.\ 60, Amer.\ Math.\ Soc., Providence, 1998.

\bibitem{Rieffel} M.A.\ Rieffel, \emph{The cancellation theorem for projective modules over irrational rotation \Cstar-algebras}, Proc.\ London Math.\ Soc.\ {\bf 47} (1983), 285--302.

\bibitem {T1} J.\ Tomiyama, \emph{Invitation to \Cstar-algebras and topological dynamical systems}, World Scientific Publishing Co., Singapore, 1987.

\bibitem {Tnotes1} J.\ Tomiyama, \emph{The interplay between topological dynamics
and theory of {$C^*$}-algebras}, Lecture Notes Series, Vol.\ 2, Seoul National
University Research Institute of Mathematics Global Analysis Research Center,
Seoul, 1992.

\bibitem {T3} J.\ Tomiyama, \emph{Structure of ideals and isomorphisms of \Cstar-crossed products by single homeomorphisms}, Tokyo J.\ Math.\ {\bf 23} (2000), 1--13.

\bibitem {T4} J.\ Tomiyama, \emph{Hulls and kernels from topological dynamical systems and their applications to homeomorphism \Cstar-algebras}, J.\ Math.\ Soc.\ Japan {\bf 56} (2004), 349--364.

\end{thebibliography}
\end{document}